\documentclass[leqno,12pt]{article} 
\usepackage{color}
\usepackage{pb-diagram}

\usepackage{amsmath}
\usepackage{amssymb}
\usepackage{amsthm}
\usepackage{ascmac}

\newtheorem{theorem}{Theorem}[section]
\newtheorem{proposition}[theorem]{Proposition}
\newtheorem{lemma}[theorem]{Lemma}
\newtheorem{corollary}[theorem]{Corollary}
\newtheorem{remark}[theorem]{Remark}
\newtheorem{example}[theorem]{Example}
\newtheorem{definition}[theorem]{Definition}

\begin{document}
\title{The quaternionic/hypercomplex-correspondence}
\author{%
Vicente Cort{\' e}s
and
Kazuyuki Hasegawa
}

\maketitle
\begin{abstract}
Given a quaternionic manifold $M$ 
with a certain $\mathrm{U}(1)$-symmetry, we construct
a hypercomplex manifold $M'$ of the same dimension. 
This construction generalizes 
the quaternionic K\"ahler/hyper-K\"ahler-correspondence. 
As an example of this construction, 
we obtain a compact homogeneous hypercomplex manifold which does not admit any 
hyper-K\"ahler structure. Therefore our construction is a  proper generalization 
of the quaternionic K\"ahler/hyper-K\"ahler-correspondence. 

%
\end{abstract}

\footnote[0]{
{\it 2010 Mathematics Subject Classification}. 
Primary 53C10. Secondary 53C56, 53C26.}
\footnote[0]{
{\it Keywords and phrases.} quaternionic manifold, 
hypercomplex manifold, hypercomplex moment map, 
quaternionic/hypercomplex-correspondence. 
}

\section{Introduction}
\setcounter{equation}{0}

Let us recall that there exist constructions due to Andriy Haydys, called the QK/HK-correspondence and the HK/QK-correspondence,
which relate quaternionic K\"ahler manifolds to hyper-K\"ahler manifolds of the same dimension \cite{Haydys}. These constructions
have been generalized to include possibly indefinite metrics 
\cite{ACM, ACDM}. In this way the supergravity c-map metric and a one-parameter deformation thereof have been described as an application of the HK/QK-correspondence with indefinite initial hyper-K\"ahler data. 
Many complete quaternionic K\"ahler manifolds can be obtained in this way, see for instance \cite{CDJL} for co-homogeneity one examples.

The main result of this paper, see Theorem \ref{qh_correspond}, is a construction of a hypercomplex manifold 
from a quaternionic manifold with a ${\rm U}(1)$-action, 
which we may call the {\it quaternionic/hy\-per\-com\-plex\--correspondence} 
(Q/H-correspondence for short). This construction generalizes the 
QK/HK-correspondence. 

In \cite{S,J,PPS}, it is shown that with every 
quaternionic manifold $M$ 
one can associate an ${\mathbb H}^*/\{ \pm 1\}$-bundle 
over $M$ and a 
hypercomplex structure on the total space of the bundle. 
More precisely  \cite{PPS}, there exists  a one-parameter family of 
${\mathbb H}^*/\{ \pm 1\}$-bundles such that, given a quaternionic connection on 
$M$, each of the bundles is endowed with an almost hypercomplex structure. 
For a particular choice of the parameter, 
the almost hypercomplex structure is integrable and independent 
of the connection. 
Here we will adopt a different point of view. 
Instead of a one-parameter family of bundles,  
we will define a single principal ${\mathbb H}^*/\{ \pm 1\}$-bundle, 
which we call 
the \emph{Swann bundle}, endowed with a 
one-parameter family of almost hypercomplex structures (still depending on a quaternionic connection). Again we 
find that, for a particular choice of the parameter, namely $c=-4(n+1)$, the almost hypercomplex structure is always integrable and independent 
of the connection, see Proposition \ref{dependence}. Here $4n=\dim M$. 
For all other values of the parameter, we show that  the almost hypercomplex structure is integrable 
if and only if all $I\in Q$, where $Q$ denotes 
the quaternionic structure, are skew-symmetric with respect to the skew-symmetric part of the Ricci-curvature, see Theorem
\ref{hypercomlex_cone}.

Now we briefly explain how we obtain the Q/H-correspondence. 
Given an infinitesimal automorphism 
$X$ of a quaternionic manifold $(M,Q,\nabla)$ endowed with a 
quaternionic connection $\nabla$, 
we show that the natural lift $\hat{X}$ of $X$ to the Swann bundle $\hat{M}$
preserves each member 
of the one-parameter family of almost hypercomplex structures. 
The next step is to perform a hypercomplex reduction with respect to $\hat{X}$. 
Recall that hypercomplex reduction 
was introduced by Dominic Joyce  in \cite{J}. It is defined as the quotient of a level set 
of a moment map by the group action. 
The construction is based on the notion of a moment map in this context as defined 
in  \cite{J}. Here we define the moment map 
for the infinitesimal automorphism $\hat{X}$ by the equation (\ref{mm:eq})
and analyse Joyce's conditions 
in Proposition \ref{moment_map}. Assuming that $\hat{X}$ generates a free 
$\mathrm{U}(1)$-action, we can finally perform the reduction obtaining a 
hypercomplex manifold $M'$. 
Otherwise, we can construct the hypercomplex structure 
on a submanifold transversal to the 
foliation defined by $\hat{X}$ (on some open submanifold of $\hat{M}$).

Examples of our Q/H-correspondence include compact homogeneous hypercomplex manifolds. 
Indeed, starting with a homogeneous quaternionic 
Hopf manifold  
\[ ({\mathbb R}^{>0} / \langle \lambda \rangle) \times
    \frac{{\rm Sp}(n) {\rm U}(1)}{{\rm Sp}(n-1) \triangle_{{\rm U}(1)}}, 
\]
we obtain a homogeneous hypercomplex Hopf manifold
\[ ({\mathbb R}^{>0} / \langle \lambda \rangle) \times
    \frac{{\rm Sp}(n)}{{\rm Sp}(n-1)}
\]
by the Q/H-correspondence, see Example \ref{Hopf}. 
Note that this hypercomplex manifold 
does not admit any hyper-K{\"a}hler structure  
for topological reasons. 
Therefore our construction is a  proper generalization 
of the QK/HK-correspondence. 

\section{Preliminaries}
\setcounter{equation}{0}

Throughout this paper, 
all manifolds are assumed to be smooth and without boundary 
and maps are assumed to be smooth 
unless otherwise mentioned. 
The space of sections of a vector bundle $E\rightarrow M$ 
is denoted by $\Gamma(E)$. 

We say that $M$ is a {\it quaternionic 
manifold} with the quaternionic structure $Q$ if 
$Q$ is a subbundle of ${\rm End}(T M)$ 
of rank $3$ which at every point $x\in M$ is spanned by endomorphisms  
$I_{1}$, $I_{2}$, $I_{3}\in \mathrm{End} (T_xM)$ satisfying
\begin{eqnarray}\label{quaternionic}
I_{1}^{2}=I_{2}^{2}=I_{3}^{2}=-\mathrm{id}, \,\, I_{1} I_{2}=-I_{2}I_{1}=I_{3}, 
\end{eqnarray}
and there exists a torsion-free connection $\nabla$ on $M$ 
such that $\nabla$ preserves $Q$, that is, 
$\nabla_X \Gamma(Q) \subset \Gamma(Q)$ for all $X\in\Gamma (TM)$. 
Note that we use the same letter $\nabla$ for the connection on ${\rm End}(T M)$ 
induced by $\nabla$ if there is no confusion. 
Such a torsion-free connection 
$\nabla$ is called a {\it quaternionic connection} and 
the triplet $(I_{1}$, $I_{2}$, $I_{3})$ is called an 
{\it admissible frame} of $Q$ at $x$. 
The dimension of a quaternionic manifold $M$ is denoted by $4n$. 
Note that a quaternionic connection is not unique, in fact, 
there is the following result \cite{Fu,AM1}. 

\begin{lemma}\label{q_conn_eq}
Let $\nabla^{1}$ and $\nabla^{2}$ be quaternionic connections on $(M,Q)$. Then there exists a 1-form $\xi$ on $M$ such that 
\begin{eqnarray}\label{q_conn}
\nabla^{2}_{X} Y = \nabla^{1}_{X} Y +S^{\xi}_{X} Y
\end{eqnarray}
for all $X$, $Y \in \Gamma(T M)$, where $S^{\xi}$
is defined by 
\begin{align*}
S^{\xi}_{X}Y =\; 
& \xi(X)Y + \xi(Y)X - \xi(I_{1}X)I_{1}Y - \xi(I_{1}Y)I_{1}X \\
   & - \xi(I_{2}X)I_{2}Y - \xi(I_{2}Y)I_{2}X - \xi(I_{3}X)I_{3}Y - \xi(I_{3}Y)I_{3}X. 
\end{align*}
Conversely, for a given quaternionic connection $\nabla^{1}$, the connection 
$\nabla^{2}$ given by the equation above is also a quaternionic connection. 
\end{lemma}

An {\it almost hypercomplex manifold} is defined to be a manifold $M$ endowed 
with 3 almost complex structures 
$I_{1}$, $I_{2}$, $I_{3}$ satisfying the quaternionic relations (\ref{quaternionic}). 
If $I_{1}$, $I_{2}$, $I_{3}$ are integrable, then $M$ is called 
a {\it hypercomplex manifold}. 
There exists a unique torsion-free connection on a hypercomplex manifold 
for which the hypercomplex structures are parallel. 
It is called the {\it Obata connection} \cite{O}. 
Obviously, hypercomplex manifolds are quaternionic manifolds with 
$Q=\langle I_{1},I_{2},I_{3} \rangle$.



\section{The canonical family of almost hypercomplex structures on 
the Swann bundle $\hat{M}$}
\setcounter{equation}{0}
In this section we will define a 
principal ${\mathbb R}^{>0} \times {\rm SO}(3)$-bundle $\hat{M}\rightarrow M$ over a quaternionic manifold 
$(M,Q)$ equipped with a quaternionic connection $\nabla$ and endow $\hat{M}$ with a one-parameter family of almost hypercomplex structures depending on the quaternionic connection $\nabla$. 
Then we will study the integrability of the hypercomplex structure and its dependence (or independence) on the choice of $\nabla$ for different values of the parameter. 
\subsection{The principal bundle $\hat{M}\rightarrow M$}
Let $S$ be the principal ${\rm SO(3)}$-bundle of admissible frames 
 $(I_{1}$, $I_{2}$, $I_{3})$ over a quaternionic manifold $(M,Q)$. 
The principal action $\tau$ of $g \in {\rm SO}(3)$ 
is given by $\tau(s,g)=sg^{\varepsilon}$
for $s=(I_{1},I_{2},I_{3}) \in S$, where $\varepsilon=1$ (resp.  $\varepsilon=-1$) 
if $S$ is considered as a right (resp. left)-principal bundle. 
The bundle projection of $S$ is denoted by $\pi_{S}$. 
We take a basis $(e_{1},e_{2},e_{3})$ of 
${\mathbb R}^{3} \cong {\rm Im}\, {\mathbb H} \cong \mathfrak{sp}(1)
\cong \mathfrak{so}(3)$ so that  
\[ [e_{\alpha},e_{\beta}]=2e_{\gamma}  \] 
for any cyclic permutation $(\alpha,\beta,\gamma)$. 
Hereafter $(\alpha,\beta,\gamma)$ 
will be always a cyclic permutation, whenever 
the three letters appear in an expression. 
A quaternionic connection induces a principal connection 
$\theta:TS \to \mathfrak{so}(3)$ and we denote 
$\theta=\sum \theta^{\alpha}e_{\alpha}$. 
Moreover we consider the principal ${\mathbb R}^{>0}$-bundle 
$S_{0}:=(\Lambda^{4n}(T^{\ast}M) \backslash \{ 0 \})/\{ \pm 1 \}$ over $M$, 
where 
${\mathbb R}^{>0}=\{ a \in {\mathbb R} \mid a >0 \}$. 
The principal ${\mathbb R}^{>0}$-action $\tau_{0}$
on $S_{0}$ is given by scalar multiplication 
$\tau_{0} (\rho,a):= \rho a^{\varepsilon}$ ($\varepsilon=\pm 1$) 
for $\rho \in S_{0}$ and $a \in {\mathbb R}^{>0}$. 
The bundle projection of $S_{0}$ is denoted by $\pi_{S_{0}}$. 
A quaternionic connection induces 
also a principal connection
$\theta_{0}:TS_{0} \to {\mathbb R} = \mathrm{Lie}\, (\mathbb{R}^{>0})$. 
The product $S_{0} \times S$ is a principal 
${\mathbb R}^{>0} \times {\rm SO}(3)$-bundle over 
$M \times M$ whose principal action is $\tau_{0} \times \tau$.   
The ${\mathbb R}^{4} ( \cong {\mathbb R} \oplus \mathfrak{so}(3))$-valued 
$1$-form $(\theta_{0} \circ pr_{TS_{0}} ,  \theta \circ pr_{TS})
=(\theta_{0} \circ pr_{TS_{0}}, 
\theta_{1}\circ pr_{TS},\theta_{2}\circ pr_{TS},\theta_{3}\circ pr_{TS})$ 
is a principal connection on $S_{0} \times S$, where 
$pr_{TS_{0}}$ (resp. $pr_{TS}$) is the projection from 
$T(S_{0} \times S) \cong TS_{0} \times TS$ onto $TS_{0}$ (resp. $TS$). 
Note that the Lie group 
${\mathbb R}^{>0} \times {\rm SO}(3)
={\mathbb H}^{\ast}/\{ \pm 1\}$.

Let $\triangle: M \to M \times M$ be the diagonal map defined by 
$\triangle(x)=(x,x)$ for each $x \in M$. 
The pullback bundle 
\[ \hat{M}:=\triangle^{\ast} (S_{0} \times S)
= \{ (x,(\rho , s))\in M \times (S_0 \times S) \mid x= \pi_{S_0}(\rho ) = \pi_S (s)\}\] 
is a principal 
${\mathbb R}^{>0} \times {\rm SO}(3)$-bundle over 
$M$ and $\bar{\theta}:=\triangle_{\#}^{\ast}
(\theta_{0}\circ pr_{TS_{0}},\theta \circ pr_{TS})$
is a principal connection on $\hat{M}$, 
where $\triangle_{\#} : \hat{M} \to S_{0} \times S$ 
is the canonical bundle map. 
The bundle projection of $\hat{M}$ onto $M$ is denoted by $\hat{\pi}$. 
Using the bundle projections $\hat{\pi}$, $\pi_{S_{0}}$, $\pi_{S}$
and the principal connections $\bar{\theta}$, $\theta_{0}$, $\theta$, 
we have 
the decomposition 
\begin{equation}\label{hori_vert}
T \hat{M} = \bar{\cal V} \oplus \bar{\cal H}, \,\, 
T S_{0}={\cal V}_{0} \oplus {\cal H}_{0}, \,\,
T S={\cal V} \oplus {\cal H}, 
\end{equation}
where $\bar{\cal V}=\mathrm{Ker} \, \hat{\pi}_{\ast}$, 
$\bar{\cal H}=\mathrm{Ker} \, \bar{\theta}$ and so on. 
It holds that 
$(\triangle_{\#})_{\ast}(\bar{\cal V}_{(x,(\rho,s))})
=({\cal V}_{0})_{\rho} \times {\cal V}_{s}$ and 
$(\triangle_{\#})_{\ast}(\bar{\cal H}_{(x,(\rho,s))}) \subset 
({\cal H}_{0})_{\rho} \times {\cal H}_{s}$ for 
each $(x,(\rho,s)) \in \hat{M}$. 
Set $\triangle_{S}:=pr_{TS} \circ (\triangle_{\#})_{\ast}$
and $\triangle_{S_{0}}:=pr_{TS_{0}} \circ (\triangle_{\#})_{\ast}$. 
The principal actions on $\hat{M}$, $S_{0} \times S$, $S_{0}$ and $S$ 
induce fundamental vector fields. We denote by $\widetilde{A}$ the fundamental 
vector field corresponding to a Lie algebra element $A$, irrespective of the manifold on which the 
vector field is defined, and set $Z_{\alpha}=\widetilde{e}_{\alpha}$ $(\alpha=1,2,3)$.
Note that $[Z_{\alpha},Z_{\beta}]
=2 \varepsilon Z_{\gamma}$.

\subsection{The canonical family of almost hypercomplex structures}
Let $(M,Q)$ be a quaternionic manifold, $\nabla$ a quaternionic connection and $\hat{\pi} : \hat{M} \rightarrow M$ the 
principal ${\mathbb R}^{>0} \times {\rm SO}(3)$-bundle with connection $\bar\theta$ constructed in the previous subsection. 
In this subsection, 
we define a canonical family of almost 
hypercomplex structures on $\hat{M}$
and consider their integrability. 

Set $e_{0}:=1 \in  {\mathbb R}\, (\cong T_{1} {\mathbb R}^{>0} )$ 
and $Z_{0}^{c}:=c \, \widetilde{e}_{0}$ 
for a nonzero real number $c$.
We denote the horizontal lifts relative to the connections 
$\bar{\theta}$, $\theta$, $\theta_{0}$ by 
$(\,\, \cdot \,\,)^{\bar{h}}$, $(\,\, \cdot \,\,)^{h}$, 
$(\,\, \cdot \,\,)^{h_{0}}$ , respectively. 
An almost hypercomplex structure 
$(\hat{I}^{\bar{\theta}, c}_{1},\hat{I}^{\bar{\theta}, c}_{2},
\hat{I}^{\bar{\theta}, c}_{3})$ 
on $\hat{M}$ is defined by 
\[ \hat{I}^{\bar{\theta}, c}_{\alpha}Z^{c}_{0}=- Z_{\alpha},\quad  
\hat{I}^{\bar{\theta}, c}_{\alpha}Z_{\alpha}= Z_{0}^{c}, \quad  
\hat{I}^{\bar{\theta}, c}_{\alpha}Z_{\beta}=Z_{\gamma}, \quad  
\hat{I}^{\bar{\theta}, c}_{\alpha}Z_{\gamma}=-Z_{\beta} \]
and
\[ (\hat{I}^{\bar{\theta}, c}_{\alpha})_{(x,(\rho,s))} (X)
=(I_{\alpha} (\hat{\pi}_{\ast} X))^{\bar{h}}_{(x,(\rho,s))} \]
for all horizontal vector $X$ at $(x,(\rho,s)) \in \hat{M}$, 
where $s=(I_{1},I_{2},I_{3})$. 
Note that the triple 
$(\hat{I}^{\bar{\theta}, c}_{1},\hat{I}^{\bar{\theta}, c}_{2},\hat{I}^{\bar{\theta}, c}_{3})$
depends on the connection form $\bar{\theta}$
and $c$. 

\begin{lemma}\label{component}
For any horizontal lift $X^{\bar{h}} \in \bar{\cal H}_{(x,(\rho,s))}$ at 
$(x,(\rho,s)) \in \hat{M}$, we have
$(\triangle_{\#})_{\ast} X^{\bar{h}}$ $= (X^{h_{0}}_{\rho},X^{h}_{s})$. 
In particular, it holds
$(\triangle_{\#})_{\ast}
((\hat{I}^{\bar{\theta}, c}_{\alpha})_{(x,(\rho,s))} (X^{\bar{h}}))
=((I_{\alpha} X)^{h_{0}}_{\rho}
,(I_{\alpha} X)^{h}_{s}), 
$
where $s=(I_{1},I_{2},I_{3})$. 
As a consequence, the horizontal lift $X^{\bar{h}}$ of a vector field $X$ on $M$
is $\triangle_\#$-related to the vector field $(X^{h_0},X^{h})$, which is 
the horizontal lift of $(X,X)${\rm :} 
\[ (\triangle_\#)_\ast X^{\bar{h}} = (X^{h_0},X^{h}) \circ \triangle_\#.\] 
\end{lemma}

\begin{proof}
$(\triangle_{\#})_{\ast} X^{\bar{h}}$ and
$(X^{h_{0}},X^{h})$ are horizontal vectors of $S_{0} \times S$, 
since applying the connection form $(\theta_{0} \circ pr_{TS_{0}}, \theta \circ pr_{TS})$ on 
both vectors gives zero. On the other hand, applying  
$(\pi_{S_{0}} \times \pi_{S})_{\ast}$ on both vectors gives $(X,X)$
because of $(\pi_{S_{0}} \times \pi_{S}) 
\circ \triangle_{\#}=\triangle \circ \hat{\pi}$. 
This proves 
$(\triangle_{\#})_{\ast} X^{\bar{h}}=(X^{h_{0}},X^{h}) \circ \triangle_{\#}$. 
Now it is easy to obtain $(\triangle_{\#})_{\ast}
((\hat{I}^{\bar{\theta}, c}_{\alpha})_{(x,(\rho,s))} (X^{\bar{h}}))
=((I_{\alpha} X)^{h_{0}}_{\rho},(I_{\alpha} X)^{h}_{s}) 
$ using that $(\hat{I}^{\bar{\theta}, c}_{\alpha})_{(x,(\rho,s))} (X^{\bar{h}})
=(I_{\alpha} X)^{\bar{h}}$.  
\end{proof}

\begin{lemma}\label{dependence_hypercomplex}
Let $\nabla^{1}$ and $\nabla^{2}=\nabla^{1}+S^{\xi}$ be quaternionic connections on $(M,Q)$, where 
$\xi \in \Gamma (T^*M)$. We denote the almost hypercomplex structure 
defined above with respect to $\nabla^{i}$ $(i=1,2)$ and $c \neq 0$ by 
$\hat{I}_{\alpha}^{i,c}$ $(\alpha=1,2,3)$. Then we have
\[
   \hat{I}_{\alpha}^{1,c}-\hat{I}_{\alpha}^{2,c}
=\varepsilon \left( 1  
   + \frac{4(n+1)}{c} \right) 
   \left( (\hat{\pi}^{\ast} \xi) \otimes Z_{\alpha} 
+  ((\hat{\pi}^{\ast} (\xi \circ I_{\alpha}))  \otimes Z_{0}^{c} \right)
\]
at each point $(x,(\rho, s)) \in \hat{M}$, where $s=(I_{1},I_{2},I_{3})$. 
\end{lemma}
\begin{proof}
We consider any point $(x,(\rho,s)) \in \hat{M}$, $s=(I_{1},I_{2},I_{3})$,  
and omit the reference point in the proof. 
The corresponding connection forms induced by $\nabla^{i}$ are denoted by 
$\bar{\theta}^{i}$, $\theta^{i}=(\theta^{i}_{1},\theta^{i}_{2},\theta^{i}_{3})$, 
$\theta_{0}^{i}$ $(i=1,2)$, respectively. 
The tangent bundle $T \hat{M}$ is decomposed into 
$T\hat{M}=\bar{\cal V} \oplus \bar{\cal H}^{1} 
=\bar{\cal V} \oplus \bar{\cal H}^{2}$, 
where $\bar{\cal H}^{i}=\mathrm{Ker}\, \bar{\theta}^{i}$. We express 
any tangent vector $X$ of $\hat{M}$ as 
\begin{eqnarray*}
X = Y^{\bar{h}_{i}} 
           + \sum_{\delta=1}^{3} a_{\delta}^{i} Z_{\delta}
           + b^{i} Z_{0}^{c}, 
\end{eqnarray*}
where $Y\in TM$. 
By the definition of $\hat{I}_{\alpha}^{i,c}$, we see
\[ \hat{I}_{\alpha}^{i,c}(X)
   = (I_{\alpha }Y)^{\bar{h}_{i}} 
           + a_{\alpha}^{i} Z_{0}^{c}
           + a_{\beta}^{i} Z_{\gamma} - a_{\gamma}^{i} Z_{\beta}
           -  b^{i} Z_{\alpha}. \]
Since 
\begin{align*}
\bar{\theta}^{1}(X)
&=\sum_{\delta=1}^{3} a_{\delta}^{1} e_{\delta}
           + c b^{1} e_{0}
=\bar{\theta}^{1}(Y^{\bar{h}_{2}}) 
      + \sum_{\delta=1}^{3} a_{\delta}^{2} e_{\delta}
           + c b^{2} e_{0} \\
&=\sum_{\delta=1}^{3} \theta_{\delta}^{1}(\triangle_{S}Y^{\bar{h}_{2}}) e_{\delta}
    +\theta_{0}^{1}(\triangle_{S_{0}} Y^{\bar{h}_{2}}) e_{0}
          + \sum_{\delta=1}^{3} a_{\delta}^{2} e_{\delta}
           + c b^{2} e_{0},                 
\end{align*}
we have $b^{1}=b^{2}+(1/c) \theta_{0}^{1}
(\triangle_{S_{0}}Y^{\bar{h}_{2}})$ and 
$a_{\delta}^{1}=a_{\delta}^{2}+\theta^{1}_{\delta}
(\triangle_{S}Y^{\bar{h}_{2}})$ ($\delta=1,2,3$). 
Therefore it holds
\begin{align*}
\hat{I}_{\alpha}^{1,c}(X) 
=\; & (I_{\alpha }Y)^{\bar{h}_{1}} 
       + a_{\alpha}^{1} Z_{0}^{c}
       + a_{\beta}^{1} Z_{\gamma} - a_{\gamma}^{1} Z_{\beta}
       -  b^{1} Z_{\alpha} \\
=\; & (I_{\alpha }Y)^{\bar{h}_{1}} 
       + (a_{\alpha}^{2}+\theta^{1}_{\alpha}(\triangle_{S}
       Y^{\bar{h}_{2}}) ) Z_{0}^{c}
       + (a_{\beta}^{2} +\theta^{1}_{\beta}(\triangle_{S}
       Y^{\bar{h}_{2}}) ) Z_{\gamma} \\
& - (a_{\gamma}^{2} +\theta^{1}_{\gamma}
  (\triangle_{S} Y^{\bar{h}_{2}}) ) Z_{\beta}
       -  (b^{2} +(1/c) \theta_{0}^{1}
       (\triangle_{S_{0}} Y^{\bar{h}_{2}}) )Z_{\alpha} \\     
=\; & (I_{\alpha }Y)^{\bar{h}_{1}} - (I_{\alpha }Y)^{\bar{h}_{2}} 
+ \hat{I}_{\alpha}^{2,c}(X) \\
  & + \theta^{1}_{\alpha}(\triangle_{S} Y^{\bar{h}_{2}}) Z_{0}^{c}
       + \theta^{1}_{\beta}(\triangle_{S} Y^{\bar{h}_{2}}) Z_{\gamma} 
       - \theta^{1}_{\gamma}(\triangle_{S} Y^{\bar{h}_{2}}) Z_{\beta}
       - (1/c) \theta_{0}^{1}(\triangle_{S_{0}} Y^{\bar{h}_{2}} )Z_{\alpha}. 
\end{align*}
Let $s_{0}:U \to S_{0}$ and $s:U \to S$ 
be local sections defined on an open set $U$ in $M$. 
Then $\bar{s}:=(s_{0},s) \circ \triangle$ is a local section of $\hat{M}$. 
The pull backs of 
$\theta^{i}$, 
$\theta_{0}^{i}$ to $U$ are denoted by $\theta^{i,U}$ and $\theta_{0}^{i,U}$. 
If we define the one forms $\theta^{i,U}_{\alpha}$ by 
$\theta^{i,U}=s^{\ast} \theta^{i}
=(1/2) \sum (\theta^{i,U}_{\alpha}) e_{\alpha}$. 
From Lemma \ref{q_conn_eq} and 
\[ 
\nabla^{i} I_{\alpha}
=\varepsilon 
( \theta^{i,U}_{\gamma} \otimes I_{\beta}
-\theta^{i,U}_{\beta} \otimes I_{\gamma}) \quad (i=1,2)
\]
one can check that 
\begin{eqnarray} \label{theta2-theta1}
\theta^{2,U}_{\delta} - \theta^{1,U}_{\delta} 
&=&- 2 \varepsilon (\xi \circ I_{\delta}) \quad  
(\delta=1,2,3), \\  
\theta^{2,U}_{0} - \theta^{1,U}_{0} &=& 4 \varepsilon (n+1)  \xi. \nonumber
\end{eqnarray}
It is easy to see that
\begin{eqnarray}\label{Y1-Y2}
       Y^{\bar{h}_{1}}-Y^{\bar{h}_{2}} 
&=& \bar{s}_{\ast}(Y)-v_{1}(\bar{s}_{\ast}(Y)) 
- \bar{s}_{\ast}(Y) + v_{2}(\bar{s}_{\ast}(Y))  \nonumber\\
&=& - \sum_{\delta=1}^{3} 
(\theta^{1}_{\delta} (s_{\ast}Y) - \theta^{2}_{\delta} (s_{\ast}Y) )Z_{\delta}
       -(1/c)  (\theta^{1}_{0} (s_{0 \ast}Y) 
       - \theta^{2}_{0} (s_{0 \ast}Y)) Z_{0}^{c}  \nonumber\\
&=& - \frac{1}{2}\sum_{\delta=1}^{3} (\theta^{1,U}_{\delta} (Y) - \theta^{2,U}_{\delta} (Y) )Z_{\delta}
       -(1/c)  (\theta^{1,U}_{0} (Y) - \theta^{2,U}_{0} (Y)) Z_{0}^{c} \nonumber\\
       &\stackrel{(\ref{theta2-theta1})}{=}& 
       - \varepsilon \sum_{\delta=1}^{3} \xi(I_{\delta}Y) Z_{\delta} 
+ \frac{4 \varepsilon (n+1)}{c} \xi(Y)Z_{0}^{c} ,  
\end{eqnarray}
where $v_{i}:T \hat{M} \to \bar{\cal V}$ 
is the projection with respect to $\bar{\theta}^{i}$ $(i=1,2)$.
Finally we obtain 
\begin{align*}
  \hat{I}_{\alpha}^{1,c}(X) &-\hat{I}_{\alpha}^{2,c}(X) \\ 
  =&  - \varepsilon 
  \sum_{\delta=1}^{3} \xi(I_{\delta}I_{\alpha}Y) Z_{\delta} 
+ \frac{4 \varepsilon (n+1)}{c} \xi(I_{\alpha}Y)Z_{0}^{c} \\
&+ \theta^{1}_{\alpha}(\triangle_{S} Y^{\bar{h}_{2}}) Z_{0}^{c}
       + \theta^{1}_{\beta}(\triangle_{S} Y^{\bar{h}_{2}}) Z_{\gamma} 
       - \theta^{1}_{\gamma}(\triangle_{S} Y^{\bar{h}_{2}}) Z_{\beta}
       - (1/c) \theta_{0}^{1}(\triangle_{S_{0}} Y^{\bar{h}_{2}} )Z_{\alpha}\\
\stackrel{(*)}{=}&- \varepsilon 
\sum_{\delta=1}^{3} \xi(I_{\delta}I_{\alpha}Y) Z_{\delta} 
+ \frac{4 \varepsilon (n+1)}{c} \xi(I_{\alpha}Y)Z_{0}^{c} \\
  &+  \varepsilon \xi (I_{\alpha}Y) Z_{0}^{c} 
      +  \varepsilon \xi(I_{\beta}Y)Z_{\gamma} 
      -  \varepsilon \xi(I_{\gamma}Y)Z_{\beta} 
      +\frac{4 \varepsilon (n+1)}{c} \xi(Y) Z_{\alpha} \\
=& \left( \varepsilon \xi(Y) 
      + \frac{4(n+1) \varepsilon }{c} \xi(Y) \right) Z_{\alpha} 
      + \left( \varepsilon \xi(I_{\alpha}Y)  
      + \frac{4(n+1)\varepsilon }{c} \xi (I_{\alpha}Y) \right) Z_{0}^{c}, 
\end{align*}
where in the step $(*)$ of the calculation we have computed
\[ \theta_\alpha^1(\triangle_SY^{\bar{h}_2}) = \theta_\alpha^1(Y^{h_2}) = \theta_\alpha^1(Y^{h_2}-Y^{h_1}) \stackrel{(\ref{Y1-Y2})}{=} 
\varepsilon \xi (I_\alpha Y)
\]
and similarly for the other terms. 
\end{proof}

The following proposition is an immediate consequence of Lemma \ref{dependence_hypercomplex}, 
cf.\ the result with \cite[Proposition 3.3]{PPS}. 

\begin{proposition}\label{dependence}
The almost hypercomplex structure 
is independent of the choice of quaternionic connection if and only if 
$c=- 4 (n+1)$.
\end{proposition}

Next we investigate transformation properties of the structures $\hat{I}_{\alpha}^{\bar{\theta},c}$ $(\alpha=1,2,3)$ under
the principal action.

\begin{lemma}\label{Lie_der_hyp_cpx}
We have $L_{Z_0}\hat{I}^{\bar{\theta}, c}_{\alpha}=L_{Z_{\alpha}} \hat{I}^{\bar{\theta}, c}_{\alpha}=0$, 
$L_{Z_{\alpha}} \hat{I}^{\bar{\theta}, c}_{\beta}
=2 \varepsilon \hat{I}^{\bar{\theta}, c}_{\gamma}$ and  
$L_{Z_{\alpha}} \hat{I}^{\bar{\theta}, c}_{\gamma}
=-2 \varepsilon \hat{I}^{\bar{\theta}, c}_{\beta}$.
\end{lemma}
\begin{proof}
Note first that the principal action generated by the vector fields $Z_a$, $a=0,\ldots , 3$, preserves the horizontal and vertical distributions. 
Moreover, the central vector field $Z_0$ commutes with the principal action and thus 
preserves the three canonical almost complex structures $\hat{I}^{\bar{\theta}, c}_{\alpha}$. 

Next we observe that it is easy to check the above equations on the vertical distribution by evaluating them on 
$Z_{0}^{c}, \dots, Z_{3}$. So it only remains to check them on the horizontal distribution. 
Let $\{ \phi_{t} \}_{t \in \mathbb{R}}$ be the flow of $Z_{1}$. 
Since 
\[ \phi_{t}((x,(\rho,s)))
=(x,(\rho,(I_{1}, (\cos 2\varepsilon t)I_{2}+(\sin 2\varepsilon t)I_{3}, 
(-\sin 2 \varepsilon t)I_{2}+(\cos 2 \varepsilon t)I_{3}))) \]
for $(x,(\rho,s)) \in \hat{M}$, where $s=(I_{1},I_{2},I_{3})$ 
and the horizontal lift of any vector field or tangent vector of $M$ is invariant 
under $\phi_{t}$, we have  
\begin{align*}
       (L_{Z_{1}} \hat{I}^{\bar{\theta}, c}_{2})_{(x,(\rho,s))}(Y^{h}) 
&= [Z_{1},\hat{I}^{\bar{\theta}, c}_{2}Y^{h}]_{(x,(\rho,s))} \\
&= \left. \frac{d}{dt} \phi_{t \ast}^{-1} 
((\hat{I}^{\bar{\theta}, c}_{2}Y^{h})_{\phi_{t}((x,(\rho,s)))}) \right|_{t=0} \\
&= \left. \frac{d}{dt} \phi_{t \ast}^{-1} 
               ((\cos 2\varepsilon t)(I_{2}Y)^{h}_{\phi_{t}((x,(\rho,s)))}
               +(\sin 2\varepsilon t)(I_{3}Y)^{h}_{\phi_{t}((x,(\rho,s)))}) \right|_{t=0} \\
&= \left. \frac{d}{dt} \left( (\cos 2\varepsilon t)(I_{2}Y)^{h}_{(x,(\rho,s))}
+(\sin 2 \varepsilon t)(I_{3}Y)^{h}_{(x,(\rho,s))} \right) \right|_{t=0} \\
&= 2 \varepsilon (I_{3}Y)^{h}_{(x,(\rho,s))}
= 2 \varepsilon (\hat{I}^{\bar{\theta}, c}_{3})_{(x,(\rho,s))} Y^{h} 
\end{align*}
and similarly $L_{Z_{1}} \hat{I}^{\bar{\theta}, c}_{1}=0$, which imply 
$L_{Z_{1}} \hat{I}^{\bar{\theta}, c}_{3}=-2 \varepsilon \hat{I}^{\bar{\theta}, c}_{2}$. 
\end{proof}

The Nijenhuis tensor for $\hat{I}^{\bar{\theta}, c}_{\alpha}$ 
is given by
\[ N^{\alpha}(U,V)=[U,V]
+\hat{I}^{\bar{\theta}, c}_{\alpha}[\hat{I}^{\bar{\theta}, c}_{\alpha}U,V]
+\hat{I}^{\bar{\theta}, c}_{\alpha}[U,\hat{I}^{\bar{\theta}, c}_{\alpha}V]
-[\hat{I}^{\bar{\theta}, c}_{\alpha}U,\hat{I}^{\bar{\theta}, c}_{\alpha}V] \]
for $U$, $V \in \Gamma(T \hat{M})$. 
Let $\bar{\Omega}$ (resp. $\Omega$) be the curvature form of 
$\bar{\theta}$ (resp. $\theta$). Take a local section $s:U \to S$ defined 
an open set $U$ of $M$. 
The pull back of  $\Omega=\sum_{\alpha=1}^3 \Omega_\alpha e_\alpha$ by $s$ is denoted by $\Omega^{U}$. 
Since the curvature form is horizontal, we have 
\begin{eqnarray}\label{curvatute_holizontal}
\varepsilon \Omega|_{s(U)}= \pi_{S}^{\ast} \Omega^{U}|_{s(U)}. 
\end{eqnarray}
If we define the two-forms $\Omega_\alpha^U$ by  $\Omega^{U} =(1/2) \sum \Omega_\alpha^Ue_\alpha$ and denote by 
$\bar{\nabla}$ the connection on $Q$ 
induced by $\nabla$, 
then we  have \[ R^{\bar{\nabla}}_{X,Y} I_{\alpha}
=[R^{\nabla}_{X,Y},I_{\alpha}] = \left[ \frac12 \sum \Omega_\delta^U (X,Y)I_\delta,I_\alpha\right] 
=\Omega^{U}_{\gamma}(X,Y)I_{\beta}- \Omega^{U}_{\beta}(X,Y)I_{\gamma}, \] 
which implies 
\begin{eqnarray}\label{curvature_form}
\Omega^{U}_{\alpha} (X,Y) 
= -\frac{1}{2n}{\rm Tr}I_{\alpha} R^{\nabla}_{X, Y} 
\end{eqnarray}
for $X$, $Y \in TM$. 
In fact, multiplying the equation 
$R^{\nabla}_{X,Y} \circ I_{\alpha} 
-I_{\alpha} \circ R^{\nabla}_{X,Y} =
\Omega^{U}_{\gamma}(X,Y)I_{\beta}- \Omega^{U}_{\beta}(X,Y)I_{\gamma}$ with $I_\beta$,  
we obtain 
\[ I_{\beta} \circ R^{\nabla}_{X,Y} \circ I_{\alpha} 
+ I_{\gamma} \circ R^{\nabla}_{X,Y} 
=-\Omega^{U}_{\gamma}(X,Y) \mathrm{id} - \Omega^{U}_{\beta}(X,Y)I_{\alpha}. \]
Taking the trace proves (\ref{curvature_form}).  
Let $Ric^{\nabla}$ be the Ricci curvature of $\nabla$ and 
its symmetric (resp.\ anti-symmetric) part is denoted by 
$(Ric^{\nabla})^{s}$ (resp.\ ($Ric^{\nabla})^{a}$). 
The Nijenhuis tensors of the canonical almost complex structures on the bundle $\hat{M}$ over the quaternionic manifold 
$(M,Q,\nabla)$ are computed in the next lemma.

\begin{lemma}\label{Nijenhuis}
If $n>1$ or $Q$ is anti-self-dual provided $n=1$, 
we have
\begin{align}
 N^{\alpha}(& Z_{0}^{c},Z_{i}) = 0 \,\,\, \mbox{for} \,\,\,  1 \leq i \leq 3, \label{vv0} \\
 N^{\alpha}( & Z_{i},Z_{j}) = 0 \,\,\, \mbox{for} \,\,\,  1 \leq i,j \leq 3, \label{vv} \\
 N^{\alpha}( & Z_{0}^{c},X^{\bar{h}}) = 0,  \label{0vh} \\
 N^{\alpha}( & Z_{i},X^{\bar{h}}) = 0 
 \,\,\, \mbox{for} \,\,\,  1 \leq i \leq 3, \label{vh} \\
  \bar{\theta}( N^{\alpha} (X^{\bar{h}}, &Y^{\bar{h}})_{(x, (\rho, s))}) \label{vhh}  \\
&=  
\frac{4 \varepsilon (n+1)+  \varepsilon c}{ 2(n+1)} 
        \left( (Ric^{\nabla})^{a}(X,Y) 
        - (Ric^{\nabla})^{a}(I_{\alpha}X,I_{\alpha}Y) \right) e_{0} \nonumber \\
 &      -\frac{4 \varepsilon (n+1)+ \varepsilon c}{ 2c(n+1)} 
        \left( (Ric^{\nabla})^{a}(X,I_{\alpha}Y) 
        + (Ric^{\nabla})^{a}(I_{\alpha}X,Y) \right) e_{\alpha}\,\,\, \mbox{and} \nonumber \\        
  &\hat{\pi}_{\ast}(N^{\alpha}(X^{\bar{h}},Y^{\bar{h}})) = 0 \,\,\, \mbox{for} \,\,\, X, Y \in \Gamma(TM),\label{hhh}               
\end{align}
where 
$(x,(\rho, s)) \in \hat{M}$ $(s=(I_{1},I_{2},I_{3}))$. \\ 
\end{lemma}

\begin{proof}
We write $\hat{I}_{\alpha}=\hat{I}^{\bar{\theta},c}_{\alpha}$ for simplicity 
in the proof of this lemma. 
It is easy to see that (\ref{vv0}--\ref{vh}) hold by the definition 
of the almost hypercomplex structure on $\hat{M}$ and Lemma \ref{Lie_der_hyp_cpx}. 
In fact, for example, we have
$N^{\alpha}(Z_{\beta},Z_{\gamma}) 
= [Z_\beta, Z_\gamma] + [Z_{\gamma},Z_{\beta} ] =0$
and 
$N^{\alpha}(Z_{\beta},X^{\bar{h}}) 
= \hat{I}_{\alpha}[Z_{\beta},\hat{I}_{\alpha} X^{\bar{h}}]
      -[\hat{I}_{\alpha}Z_{\beta},\hat{I}_{\alpha} X^{\bar{h}}] 
= \hat{I}_{\alpha} ( -2 \varepsilon \hat{I}_{\gamma} )(X^{\bar{h}})
- 2 \varepsilon \hat{I}_{\beta} (X^{\bar{h}}) = 0.$
The other equations are proved similarly. Next we show (\ref{vhh}). 
It holds 
$(\theta_{i} \circ \hat{I}_{\alpha})(Z^{c}_{0}) = - \delta_{i \alpha}$, 
$(\theta_{i} \circ \hat{I}_{\alpha})(Z_{\alpha}) = c \delta_{i 0}$, 
$(\theta_{i} \circ \hat{I}_{\alpha})(Z_{\beta}) = \delta_{i \gamma}$, 
$(\theta_{i} \circ \hat{I}_{\alpha})(Z_{\gamma}) = - \delta_{i \beta}$. 
Using this and Lemma \ref{component}, we have
\begin{align*}
   &  \bar{\theta} (\hat{I}_{\alpha}[\hat{I}_{\alpha}X^{\bar{h}},Y^{\bar{h}}])  \\
=& (\bar{\theta} \circ \hat{I}_{\alpha}) 
      (\sum_{i=1}^{3} \theta_{i} (\triangle_{S} 
      [\hat{I}_{\alpha}X^{\bar{h}},Y^{\bar{h}}]) Z_{i} 
      + \frac{1}{c} \theta_{0}(\triangle_{S_{0}} 
      [\hat{I}_{\alpha}X^{\bar{h}},Y^{\bar{h}}])Z^{c}_{0}) \\
=& \sum_{j=0}^{3} (\sum_{i=1}^{3} \theta_{i} 
       (\triangle_{S} [\hat{I}_{\alpha}X^{\bar{h}},Y^{\bar{h}}])
       (\theta_{j} \circ \hat{I}_{\alpha}) (Z_{i}) e_{j} 
       + \frac{1}{c} \theta_{0}
       (\triangle_{S_{0}} [\hat{I}_{\alpha}X^{\bar{h}},Y^{\bar{h}}])
       (\theta_{j} \circ \hat{I}_{\alpha}) (Z^{c}_{0}) e_{j} )\\
=& \sum_{j=0}^{3} 
       (\theta_{\alpha} (\triangle_{S}
       [\hat{I}_{\alpha}X^{\bar{h}},Y^{\bar{h}}])  c \delta_{j0}e_{j}
       +\theta_{\beta} (\triangle_{S} 
       [\hat{I}_{\alpha}X^{\bar{h}},Y^{\bar{h}}])  
       \delta_{j \gamma}e_{j} \\
 &  -\theta_{\gamma} (\triangle_{S}
      [\hat{I}_{\alpha}X^{\bar{h}},Y^{\bar{h}}])  
      \delta_{j \beta}e_{j} 
       +\frac{1}{c} \theta_{0}(\triangle_{S_{0}} 
       [\hat{I}_{\alpha}X^{\bar{h}},Y^{\bar{h}}])
                        (-\delta_{j \alpha}) e_{j})
\end{align*}
As a consequence of Lemma \ref{component}, we have
$\triangle_{S_{0}} [\hat{I}_{\alpha}X^{\bar{h}},Y^{\bar{h}}] 
= [ (I_{\alpha}X)^{h_{0}},Y^{h_{0}}] |_{\triangle_{\#}(\hat{M})}$ and 
$\triangle_{S} [\hat{I}_{\alpha}X^{\bar{h}},Y^{\bar{h}}]
=[ (I_{\alpha}X)^{h},Y^{h}] |_{\triangle_{\#}(\hat{M})}$. 
By $\bar{\Omega}=\Omega+d \theta_{0}
=\sum_{\delta=1}^{3} \Omega_{\delta} e_{\delta}+(d \theta_{0}) e_{0}$, 
it holds
\begin{align*}
\bar{\theta} (\hat{I}_{\alpha}[\hat{I}_{\alpha}X^{\bar{h}},Y^{\bar{h}}])
=& -c \Omega_{\alpha} ((I_{\alpha}X)^{h},Y^{h}) e_{0} 
      -\Omega_{\beta} ((I_{\alpha}X)^{h},Y^{h}) e_{\gamma} \\
  &    +\Omega_{\gamma} ((I_{\alpha}X)^{h},Y^{h}) e_{\beta}
      +\frac{1}{c}d \theta_{0}((I_{\alpha}X)^{h_{0}},Y^{h_{0}})e_{\alpha}.
\end{align*}
Defining
\begin{align*}
A_{\alpha}(X^{\bar{h}},Y^{\bar{h}})
=& -\Omega_{\beta}(\triangle_{S} X^{\bar{h}},\triangle_{S} Y^{\bar{h}})
    +\Omega_{\beta}(\triangle_{S} \hat{I}_{\alpha}X^{\bar{h}},
    \triangle_{S} \hat{I}_{\alpha}Y^{\bar{h}}) \\
  &+\Omega_{\gamma}(\triangle_{S} \hat{I}_{\alpha}X^{\bar{h}},\triangle_{S} Y^{\bar{h}})
        +\Omega_{\gamma}(\triangle_{S} X^{\bar{h}},\triangle_{S}\hat{I}_{\alpha} Y^{\bar{h}})
\end{align*}
for $X$, $Y \in T M$, we obtain
\begin{align*}
 & \bar{\theta}(N^{\alpha}(X^{\bar{h}},Y^{\bar{h}})) \\
=&(-d \theta_{0}(X^{h_{0}},Y^{h_{0}})
+d \theta_{0}((I_{\alpha}X)^{h_{0}},(I_{\alpha}Y)^{h_{0}})
       -c \Omega_{\alpha} (X^{h},(I_{\alpha}Y)^{h})
       -c \Omega_{\alpha} ((I_{\alpha}X)^{h},Y^{h}) e_{0} \\
  & (-\Omega_{\alpha}(X^{h},Y^{h})
       +\Omega_{\alpha}((I_{\alpha}X)^{h},(I_{\alpha}Y)^{h}) 
       +\frac{1}{c}d \theta_{0}((I_{\alpha}X)^{h_{0}},Y^{h_{0}})
       +\frac{1}{c}d \theta_{0}(X^{h_{0}},(I_{\alpha}Y)^{h_{0}}))e_{\alpha} \\
 &  +A_{\alpha}(X^{\bar{h}},Y^{\bar{h}})e_{\beta} 
 +A_{\alpha}((I_{\alpha}X)^{\bar{h}},Y^{\bar{h}})e_{\gamma}.
\end{align*}
Next we show that the coefficients of $e_{0}$ and $e_{\alpha}$ can be described by the Ricci tensor of
$\nabla$ and that the other components vanish thanks to the integrability of
the almost complex structure on the twistor space of $M$  \cite{S}. 
Set 
\begin{eqnarray}\label{defB}
B:=\frac{1}{4(n+1)} (Ric^{\nabla})^{a}+\frac{1}{4n}(Ric^{\nabla})^{s}
    -\frac{1}{2n(n+2)} \Pi_{h} (Ric^{\nabla})^{s} , 
\end{eqnarray}
where $\Pi_{h} (Ric^{\nabla})^{s}$ is the $Q$-hermitian $(0,2)$-tensor defined by
\[ (\Pi_{h} (Ric^{\nabla})^{s} )(X,Y)=\frac{1}{4}
\left( (Ric^{\nabla})^{s} (X,Y)+\sum_{i=1}^{3}(Ric^{\nabla})^{s} (I_{i}X,I_{i}Y) \right)
\]
for $X$, $Y \in TM$. By \cite{AM1}, we have
\begin{eqnarray}\label{defOmega}
\Omega_{\alpha}^{U}(X,Y)=2 (B(X,I_{\alpha}Y)-B(Y,I_{\alpha}X)). 
\end{eqnarray}
Then it holds
\begin{eqnarray*}
\Omega^{U}_{\alpha}(I_{\alpha}X,Y)+\Omega^{U}_{\alpha}(X,I_{\alpha}Y)
=-\frac{1}{n+1} \left( (Ric^{\nabla})^{a}(X,Y)-(Ric^{\nabla})^{a}(I_{\alpha}X,I_{\alpha}Y)
                      \right).
\end{eqnarray*}
Since
$\varepsilon d \theta^{U}_{0}(X,Y)
= {\rm Tr} R^{\nabla}_{X,Y} 
=- Ric^{\nabla}(X,Y)+Ric^{\nabla}(Y,X)
=-2(Ric^{\nabla})^{a}(X,Y)$ and $\Omega_\alpha (X^h,Y^h) 
= (1/2) \varepsilon \Omega_\alpha^U(X,Y)$
for all tangent vector $X$, $Y$ on $M$, 
to prove (\ref{vhh}), it is sufficient to check $A_{\alpha}=0$. 
This is related to the integrability of the almost complex structure on 
the twistor space ${\cal Z}$ of the quaternionic manifold $(M,Q)$ as we explain now.  
Recall that ${\cal Z} = \{ A\in Q \mid A^2 = -\mathrm{id}\}$. 
We set 
\[ 
R^{\nabla (0,2) I}_{X,Y}
  :=\frac{1}{4} (R^{\nabla}_{X,Y}+I R^{\nabla}_{IX,Y}
                 +I R^{\nabla}_{X,IY} - R^{\nabla}_{IX,IY})
\]
for $X$, $Y \in TM$ and $I \in {\cal Z}$. 
Then 
\begin{eqnarray}\label{curvature_form_02part}
[R^{\nabla (0,2) I}_{X,Y}, I]=0
\end{eqnarray}
for any $I \in {\cal Z}$ if $n>1$. In the case of $\dim M=4$, 
(\ref{curvature_form_02part}) holds if and only if $Q$ is anti-self-dual.
See \cite{AG} for example. 
By (\ref{curvature_form}) and (\ref{curvature_form_02part}), we have
$[R^{\nabla (0,2) I_{\alpha}}_{X,Y}, I_{\alpha}] I_{\gamma} =0$ and thus 
\begin{align*}
 0 =& 2 {\rm Tr}[R^{\nabla (0,2) I_{\alpha}}_{X,Y}, I_{\alpha}] I_{\gamma} \\
    =&  {\rm Tr} (-I_{\beta}R^{\nabla}_{X,Y}
                        +I_{\beta}R^{\nabla}_{I_{\alpha}X,I_{\alpha}Y}
                        +I_{\gamma} R^{\nabla}_{I_{\alpha}X,Y}
                        +I_{\gamma} R^{\nabla}_{X,I_{\alpha}Y}) \\
    =& 2n (\Omega^{U}_{\beta}(X,Y) 
                -  \Omega^{U}_{\beta}(I_{\alpha}X,I_{\alpha}Y)      
                -  \Omega^{U}_{\gamma}(I_{\alpha}X,Y)
                -   \Omega^{U}_{\gamma}(X,I_{\alpha}Y) ) \\
    =& -4n \varepsilon A_{\alpha}(X^{\bar{h}},Y^{\bar{h}}) 
\end{align*}
for all $X$, $Y \in TM$. This proves that $A_{\alpha}=0$. 

Since $\nabla$ is torsion-free, we have (\ref{hhh}) by the similar calculation 
for the Nijenhuis tensor of the almost complex structure on the twistor space. \\ 
\end{proof}



From Lemma \ref{Nijenhuis} (and Proposition \ref{dependence}) we obtain the following result. 

\begin{theorem}\label{hypercomlex_cone}
Let $(M,Q)$ be a quaternionic manifold and $\nabla$ a quaternionic connection.  
Let $(\hat{I}^{\bar{\theta}, c}_{1},\hat{I}^{\bar{\theta}, c}_{2},
\hat{I}^{\bar{\theta}, c}_{3})$ 
be the almost hypercomplex structure on $\hat{M}$. We assume that 
$Q$ is anti-self-dual when $n=1$. If $c=-4(n+1)$, then 
the almost hypercomplex structure is integrable 
{\rm (}and independent of $\nabla${\rm )}. 
When $c \neq - 4(n+1)$, 
the almost hypercomplex structure is integrable  
if and only if $(Ric^{\nabla})^{a}$ is $Q$-hermitian, 
that is, it is hermitian with respect to 
$I$ for all $I \in {\cal Z}$,
where ${\cal Z}$ is the twistor space of $(M,Q)$. 
\end{theorem}

We call $\hat{M}$ the {\it Swann bundle} of $M$, although 
the terminology ``Swann bundle'' is also used for the quotient space
$\hat{M} / \mathbb{Z}$ with $c= -4(n+1)$ in \cite{P}. 
From now on we will only consider the case that 
$(\hat{I}^{\bar{\theta}, c}_{1},
\hat{I}^{\bar{\theta}, c}_{2},\hat{I}^{\bar{\theta}, c}_{3})$
is a hypercomplex structure, i.e.\ integrable. 
We note that,  for each fixed quaternionic connection, 
$(\hat{I}^{\bar{\theta}, c}_{1},\hat{I}^{\bar{\theta}, c}_{2},\hat{I}^{\bar{\theta}, c}_{3})
\neq (\hat{I}^{\bar{\theta}, c^{\prime}}_{1},\hat{I}^{\bar{\theta}, 
c^{\prime}}_{2},\hat{I}^{\bar{\theta}, c^{\prime}}_{3})$
if $c \neq c^{\prime}$. 
Although it is obvious from the definition, we can also see it by considering the Obata connection. 
From Lemma \ref{Lie_der_hyp_cpx}, 
it follows that 
$\hat{\nabla}^{c}_{\widetilde{e}_{0}} \widetilde{e}_{0}
=(1/c)\hat{\nabla}^{1}_{\widetilde{e}_{0}}\widetilde{e}_{0}$, 
where $\hat{\nabla}^{c}$ is the Obata connection 
for the hypercomplex structure
$(\hat{I}^{\bar{\theta}, c}_{1},\hat{I}^{\bar{\theta}, c}_{2},\hat{I}^{\bar{\theta}, c}_{3})$. 

\section{An infinitesimal quaternionic vector field and its 
natural lift}
\setcounter{equation}{0}

A vector field $X$ on $(M,Q)$ is called {\it quaternionic}  
if its (local) flow $\varphi_{t}$ 
satisfies 
\[ 
\varphi_{-t}^{\ast} I := 
\varphi_{t \ast} \circ I \circ \varphi_{t \ast}^{-1} \in Q
\]
for all $I \in Q$ and for all $t$. 
For a connection $\nabla$ and $X \in \Gamma(TM)$, we define
\begin{eqnarray}\label{affine}
(L_{X}  \nabla)_{Y} Z
:= L_{X} (\nabla_{Y}Z)-\nabla_{L_{X}Y}Z-\nabla_{Y}(L_{X}Z), 
\end{eqnarray}
where $Y$, $Z \in \Gamma(TM)$. Note that 
$L_{X}  \nabla$ is a tensor. 
In this paper, we study $(M,Q)$ with a quaternionic vector field $X$ 
which is also affine, that is $L_{X}  \nabla=0$. 
So we start by studying the condition $L_{X}  \nabla=0$. 
We define the Hessian $H^{\nabla}$ with respect to $\nabla$ by 
\[ H^{\nabla}_{Y,Z}X=\nabla_{Y}\nabla_{Z}X-\nabla_{\nabla_{Y}Z}X \]
for $X$, $Y$, $Z \in \Gamma(TM)$. 
By similar arguments as in \cite{AM0}, we have the following.

\begin{lemma}\label{affine_eq}
Let $\nabla$ be a quaternionic connection of $(M,Q)$ 
and $X$ a quaternionic vector field. 
Then the following conditions are equivalent each other. \\
{\rm (1)} $L_{X} \nabla=0$, \\
{\rm (2)} $R^{\nabla}_{X,Y}Z = - H^{\nabla}_{Y,Z}X$ for all $Y$, $Z \in TM$, \\
{\rm (3)} $Ric^{\nabla}(X,Z)={\rm Tr}H^{\nabla}_{(\,\, \cdot \,\,) ,Z}X$
for all $Z \in TM$.
\end{lemma}

\begin{proof}
Since $X$ is a quaternionic vector field, 
$\varphi_{-t}^{\ast} \nabla$ is a quaternionic connection with respect to $Q$, where 
$\varphi_{-t}^{\ast} \nabla$ is the connection defined by
\[ (\varphi_{-t}^{\ast} \nabla)_{Y} Z
= \varphi_{t \ast} (\nabla_{\varphi_{t \ast}^{-1}Y} \varphi_{t \ast}^{-1}Z) \]
for $Y$, $Z \in \Gamma(TM)$. 
Therefore there exists a one form $\xi_{t}$ such that 
$\varphi_{t}^{\ast} \nabla - \nabla =S^{\xi_{t}}$ by Lemma \ref{q_conn_eq}. 
Then we have
\begin{align}\label{LX_eq}
      L_{X} \nabla
=\left. \frac{d}{dt} \varphi_{t}^{\ast} \nabla  \right|_{t=0}  
=\left. \frac{d}{dt}  S^{\xi_{t}} 
     \right|_{t=0}  
=S^{\xi_{X}},
\end{align}
where $\xi_{X}=(d/dt) \xi_{t} |_{t=0}$. 
On the other hand, 
by a straightforward calculation, we have
$(L_{X} \nabla)_{Y}Z=R^{\nabla}_{X,Y}Z+H^{\nabla}_{Y,Z}X$
for all $Y$, $Z \in TM$. 
Therefore, we have
\[ (L_{X} \nabla)_{Y}Z=R^{\nabla}_{X,Y}Z+H^{\nabla}_{Y,Z}X=S^{\xi_{X}}_{Z}Y. \]
It follows that $(1) \Rightarrow (2)$. 
It is also easy to see that $(2) \Rightarrow (3)$
by taking a trace. 
Since ${\rm Tr} S^{\xi_{X}}_{Z} =4(n+1)\xi_{X}(Z)$, we have 
\[ -Ric^{\nabla}(X,Z)+{\rm Tr}H^{\nabla}_{(\,\, \cdot \,\,) ,Z}X 
      =  {\rm Tr} S^{\xi_{X}}_{Z} =4(n+1)\xi_{X}(Z). \] 
If (3) holds, then we have $(1)$. 
\end{proof}

We consider the normalizer and 
\[ N(Q):=\{ A \in {\rm End}(TM) \mid [A,I] \in Q \,\,\, \mbox{for all}\,\,\,  I \in Q \} \]
and the centralizer
\[ Z(Q):=\{ A \in {\rm End}(TM) \mid [A,I]=0 \,\,\, \mbox{for all}\,\,\,  I \in Q \}. \]
Then we see $N(Q)=Q+Z(Q)=Q+\mathbb{R} \cdot \mathrm{id}+Z_{0}(Q)$, where 
$Z_{0}(Q)$ is the subspace of $Z(Q)$ of trace-free tensors \cite{AM1}. 
Let $\nabla$ be a quaternionic connection and 
$X$ a quaternionic vector field. 
Since $L_{X} I_{\alpha}=\nabla_{X} I_{\alpha}+[I_{\alpha},(\nabla X)]$, 
$\nabla X$ is an element of 
$N(Q)$. We write $\nabla X=T+T_{0}$, where 
$T \in \Gamma(Q+\mathbb{R} \cdot \mathrm{id})$ and 
$T_{0} \in \Gamma(Z_{0}(Q))$. Note that, by \cite{AM1}, we have explicitly 
\begin{align*}
       \nabla X 
=& -\frac{1}{4n} \sum_{\alpha=1}^{3} ({\rm Tr} (\nabla X)I_{\alpha})I_{\alpha} 
\,\, (\in \Gamma(Q)) \\
  & +\frac{1}{4n} ({\rm Tr} \nabla X )\mathrm{id} 
     \,\, (\in C^{\infty}(M) \mathrm{id} )\\
  & +\frac{1}{4} 
        ((\nabla X) - \sum_{\alpha=1}^{3} I_{\alpha} (\nabla X) I_{\alpha})
                       - \frac{1}{4n} ({\rm Tr} \nabla X )\mathrm{id} 
                       \,\, (\in \Gamma(Z_{0}(Q))). 
\end{align*}
So it holds
\begin{eqnarray}\label{q-compo}
T = \frac{1}{4n} \sum_{\alpha=0}^{3} 
                  \varepsilon_{\alpha} ({\rm Tr} (\nabla X)I_{\alpha}) I_{\alpha}, 
\end{eqnarray}
where $I_{0}=\mathrm{id}$ and $\varepsilon_{0}=1$, 
$\varepsilon_{1}=\varepsilon_{2}=\varepsilon_{3}=-1$. 

\begin{proposition}\label{chara_affine}
Let $\nabla$ be a quaternionic connection and 
$X$ a quaternionic vector field. 
Then $L_{X} \nabla=0$ if and only if 
$2(Ric^{\nabla})^{a}(X, \, \cdot \,)=d ({\rm Tr} (\nabla X))$. 
\end{proposition}

\begin{proof}
By the Bianchi identity, it holds 
\begin{align*}
R^{\nabla}_{X,Y} Z
&= -R^{\nabla}_{Z,X}Y-R^{\nabla}_{Y,Z}X \\
&= -R^{\nabla}_{Z,X}Y-H^{\nabla}_{Y,Z} X+H^{\nabla}_{Z,Y} X \\
&= -R^{\nabla}_{Z,X}Y-H^{\nabla}_{Y,Z} X+(\nabla_{Z} T)(Y)+(\nabla_{Z} T_{0})(Y)
\end{align*}
for all $Y$, $Z \in TM$. Then we have
$-Ric^{\nabla}(X,Z)=-{\rm Tr}R^{\nabla}_{Z,X} 
- {\rm Tr}H^{\nabla}_{(\, \cdot \,),Z} X
+ {\rm Tr} (\nabla_{Z} T)$, 
since $T_{0}$ and $\nabla_{Z} T_{0}$ are trace-free. 
Therefore, by Lemma \ref{affine_eq}, we see that
$L_{X} \nabla=0$ if and only if 
$2(Ric^{\nabla})^{a}(Z,X) + {\rm Tr} (\nabla_{Z} T) =0$. 
Finally, because $T \in \Gamma(Q+\mathbb{R} \cdot \mathrm{id})$, we obtain 
${\rm Tr} (\nabla_{Z} T)=Z {\rm Tr} (\nabla X)$ 
by (\ref{q-compo}). This implies the conclusion. 
\end{proof}


Recall that every Killing vector field is affine 
with respect of the Levi-Civita connection. 
This means that every quaternionic Killing vector field $X$ on a 
quaternionic K{\"a}hler manifold $(M,g,Q)$ is an example of an affine  
quaternionic vector field.
This can be seen also by Proposition \ref{chara_affine}.

\begin{corollary}\label{ex_affine_con}
Let $X$ be an quaternionic vector filed on a quaternionic manifold 
$(M,Q)$. If there exists a volume element $\nu$ on $M$ such that $L_{X} \nu=0$, 
then there exists a quaternionic connection $\nabla$ such that 
$L_{X} \nabla=0$. 
\end{corollary}

\begin{proof}
We can find a quaternionic connection $\nabla$ 
such that $\nabla \nu=0$ by \cite[Theorem 2.4]{AM1}. 
Then $Ric^{\nabla}$ is symmetric. 
Because $L_{X}=\nabla_{X}-(\nabla X)$ and $L_{X} \nu=0$, 
we have ${\rm Tr}(\nabla X)=0$. 
Now the conclusion follows from Proposition \ref{chara_affine}. 
\end{proof}
If $X$ is a quaternionic vector field with the flow $\{ \varphi_{t} \}$,  
then $X$ can be lifted to $\hat{X}$ on $\hat{M}$ as follows. 
We define $\hat{\varphi}_{t}:\hat{M} \to \hat{M}$ by 
\[ \hat{\varphi}_{t}((x,(\rho,s)))
=(\varphi_{t}(x), (\varphi_{-t}^{\ast} \rho, (\varphi_{-t}^{\ast} I_{1},
\varphi_{-t}^{\ast} I_{2},
\varphi_{-t}^{\ast} I_{3})
)) \]
for $(x,(\rho,s)) \in \hat{M}$, where $s=(I_{1},I_{2},I_{3})$ and define 
\[ \hat{X}_{(x,(\rho,s))} 
= \left. \frac{d}{dt} \hat{\varphi}_{t}((x,(\rho,s))) \right|_{t=0}. \] 
%
%
%
%
%
%
The vector field $\hat{X}$ on $\hat{M}$ is called the {\it natural lift} of $X$. 
Since $\hat{X}$ is invariant by the principal 
${\mathbb R}^{>0} \times {\rm SO}(3)$-action, we have the following.

\begin{lemma}\label{invariant}
Let $\hat{X}$ be the natural lift of a quaternionic vector field $X$. 
We have $[\hat{X},\widetilde{B}]=0$ for 
$B \in  {\mathbb R} \oplus \mathfrak{so}(3)$. 
\end{lemma}

Existence of $\nu \in \Gamma(S_{0})$ such that $L_{X} \nu=0$ 
(see Corollary \ref{ex_affine_con}) is related to 
the following condition for $\hat{X}$.

\begin{lemma}\label{triviality}
Let $X$ be a quaternionic vector field $X$ on $(M,Q)$. 
The following conditions are equivalent {\rm :}  \\
{\rm (1)} there exists $\nu \in \Gamma(S_{0})$ such that $L_{X} \nu=0$, \\
{\rm (2)} there exists a trivialization $S_{0} \cong M \times {\mathbb R}^{>0}$ 
such that $\hat{X}_{p} \in T_{s}S \subset T_{p}\hat{M} \cong T_{s}S \oplus {\mathbb R}$
for all $p=(x,(\rho,s)) \in \hat{M}$. 
\end{lemma}

\begin{proof}
At first, assume that (1) holds. Then $\nu$ gives  
a trivialization $S_{0} \cong M \times {\mathbb R}^{>0}$
and $\hat{M}=S \times {\mathbb R}^{>0}$. 
We denote the component of $\hat{M}$ tangent to the second factor 
by $\hat{X}^{{\mathbb R}}$. 
For any point $(x,(\rho,s)) \in \hat{M}$, 
we see that 
$\hat{X}^{{\mathbb R}}_{(x,(\rho,s))}=0$
$\iff$ 
$\hat{X}^{{\mathbb R}}_{(x,(\nu(x),s))}=0$
$\iff$
$(L_{X} \nu)_{x}=0$ by 
Lemma \ref{invariant}. 
Conversely, we can obtain the  desired section $\nu \in \Gamma(S_{0})$ by
$\nu(x)=\Phi^{-1}(x,1)$ for each $x \in M$, 
where $\Phi:S_{0} \to M \times {\mathbb R}^{>0}$ is a 
trivialization satisfying (2). 
\end{proof}

If there exists $\nu \in \Gamma(S_{0})$ such that $L_{X} \nu=0$, 
we may assume that 
$\hat{X}$ is a tangent vector field on $S$ by Lemma \ref{triviality}. 
From now on we will assume that the quaternionic vector field $X$ 
generates a free $\mathrm{U}(1)$-action. 
Since $\mathrm{U}(1)$ is compact, 
there exists a volume form $\nu$ invariant under the group action.    
This also implies that there exists 
a quaternionic connection $\nabla$ such that 
$L_{X} \nabla=0$ by Corollary \ref{ex_affine_con}.


\section{The hypercomplex moment map}
\setcounter{equation}{0}

In this section, we consider a hypercomplex moment map on the Swann bundle. 
In \cite{J}, a hypercomplex moment map is defined as follows.

\begin{definition}[\cite{J}]
{\rm 
Let $M$ be a hypercomplex manifold 
with hypercomplex structure $I_{1}$, $I_{2}$, $I_{3}$ and 
$F$ a compact Lie group acting smoothly and freely 
on $M$ preserving $I_{i}$ {\rm(}$i=1,2,3${\rm)}. 
$F$ acts on $\mathfrak{F}=\mathrm{Lie}\, F$ by 
the adjoint action. 
A vector field on $M$ induced by $f \in \mathfrak{F}$ is denoted
by $X_{f}$. 
If a triple $\mu=(\mu_{1},\mu_{2},\mu_{3})$ of 
$F$-equivariant maps $\mu_{i}:M \to \mathfrak{F}^{\ast}$ 
{\rm (}$i=1,2,3${\rm )} satisfies 
\begin{eqnarray}\label{moment_cr}
d \mu_{1} \circ I_{1}=d \mu_{2} \circ I_{2}=d \mu_{3} \circ I_{3}
\end{eqnarray}
and 
\begin{eqnarray}\label{moment_trans}
 (d \mu_{1} \circ I_{1})(X_{f}) \,\, 
\mbox{does not vanish on} \,\, M \,\,
\mbox{for any non-zero} \,\,  f \in \mathfrak{F},
\end{eqnarray}
then $\mu$ is called the {\it hypercomplex moment map} of $F$. 
The equations {\rm (}\ref{moment_cr}{\rm )} 
are called the CR {\rm (}Cauchy-Riemann{\rm )} 
equations and the condition 
{\rm (}\ref{moment_trans}{\rm )} is called the transversality condition.
}
\end{definition}

A hypercomplex moment map produces another hypercomplex manifold by 
a quotient (Proposition 3.1 in \cite{J}). 
Let $(M,Q)$ be a quaternionic manifold with  
a quaternionic connection $\nabla$ and 
an affine quaternionic vector field $X$. 
The following lemmas hold.

\begin{lemma}\label{horizontal}
If $X$ is an affine quaternionic vector field on $(M,Q,\nabla)$ and 
$\bar{\theta}$ is the principal $\mathbb{R}^{>0} \times \mathrm{SO}(3)$-connection 
on $\hat{M}$ induced by $\nabla$, 
then $L_{\hat{X}} \bar{\theta}=0$ and $L_{\hat{X}} \hat{I}^{\bar{\theta},c}_{\alpha}=0$. 
\end{lemma}

\begin{proof}
The first equation follows from the fact that $\hat{\varphi}_{t}$ preserves the horizontal distribution, 
because $\hat{\varphi}_{t}$ is induced by a local flow ${\varphi}_{t}$
of affine transformations preserving 
the quaternionic structure. Since the almost hypercomplex structure 
$(\hat{I}^{\bar{\theta},c}_{1},\hat{I}^{\bar{\theta},c}_{2},\hat{I}^{\bar{\theta},c}_{3})$ 
is canonically associated with the data $(Q,\nabla)$ on $M$, 
it is also invariant under $\hat{\varphi}_{t}$, which implies the second equation. 
\end{proof}

From now on we assume that there exists $\nu \in \Gamma(S_{0})$ 
such that $L_{X} \nu=0$. 
Then we can identify $S_{0}=M \times {\mathbb R}^{>0}$, 
$\hat{M}=S \times {\mathbb R}^{>0}$ and 
$\hat{X}$ is a tangent vector field on $S$
by Lemma \ref{triviality}. 
In the next lemma, we identify $S$ 
with the $\hat{\varphi}_{t}$-invariant submanifold 
$S \times \{1\} \subset \hat{M}=S\times {\mathbb R}^{>0}$.

\begin{lemma}\label{L_X}
Under the above assumption, $L_{\hat{X}} \theta=0$. 
Moreover $\bar{\cal H}|_{S}=\cal{H}$ if and only if 
$\nabla \nu=0$. 
\end{lemma}
\begin{proof}
The projection from ${\mathbb R} \oplus \mathfrak{so}(3)$ 
onto ${\mathbb R}$ (resp. $\mathfrak{so}(3)$) is denoted by 
$pr_{\mathbb{R}}$ (resp. $pr_{\mathfrak{so}(3)}$). 
The first statement follows from the previous lemma, since 
$pr_{\mathfrak{so}(3)} \bar{\theta}|_{S}=\theta$.
The second statement follows from 
$pr_{\mathbb{R}} \bar{\theta}|_{S}=(\nu \circ \pi_{S})^{\ast} \theta_{0}$, 
since $\nabla \nu=(\nu^{\ast} \theta_{0}) \otimes \nu$. 
\end{proof}

For $1 \in \mathbb{R} \cong T_{1} {\mathbb R}^{>0}$, 
at $\rho \in S_{0}$, we have
\[ (\widetilde{e_{0}})_{\rho}=\widetilde{1}_{\rho}
 = \left. \frac{d}{dr} \rho \exp (\varepsilon t ) \right|_{t=0} 
 = \varepsilon \rho 
 = \left. \varepsilon r \frac{\partial }{\partial r} \right|_{\rho}, \]
where $r$ is the standard coordinate on ${\mathbb R}^{>0}$. 
Let $\nabla$ be a quaternionic connection on $(M,Q)$. 
We define 1-forms  $\hat{\theta}^{c}_{\alpha}$ on $\hat{M}$ $(\alpha=1,2,3)$ by
\[ \hat{\theta}^{c}_{\alpha}|_{TS} := A r^{\frac{2}{c}} \theta_{\alpha} 
\,\,\, \mbox{and}\,\,\,  \hat{\theta}^{c}_{\alpha}(Z^{c}_{0})=0 \]
where $A\in \mathbb{R}$ is a constant. 
A symmetric tensor $\langle \theta, \theta \rangle$ is defined by 
\[ \langle \theta, \theta \rangle(Y,Z)=\sum_{i=1}^{3} 
  \theta_{i}(Y) \theta_{i}(Z) \]
for $Y$ and $Z \in T \hat{M}$
and we set
\begin{eqnarray}\label{two_forms}
 G^{c}_{\alpha}
:= -A r^{\frac{2}{c}} \Omega_{\alpha} 
(\, \cdot \, ,\hat{I}^{\bar{\theta},c}_{\alpha} \,\, \cdot  \,\, )
       +2 \varepsilon A r^{\frac{2}{c}} \langle \theta, \theta \rangle
    +\frac{2 \varepsilon A}{c^{2}} r^{\frac{2}{c}-2}(dr \otimes dr). 
\end{eqnarray} 
Note that $G^{c}_{1}|_{{\cal V} \times {\cal V}}
=G^{c}_{2}|_{{\cal V} \times \cal{V}}
=G^{c}_{3}|_{{\cal V} \times \cal{V}}$, 
that is, the vertical components of $G^{c}_{\alpha}$ 
are independent of $\alpha$.


\begin{lemma}\label{ext_der_theta}
We have 
$d \hat{\theta}^{c}_{\alpha} (Y,Z) 
= G^{c}_{\alpha}(Y,\hat{I}^{\bar{\theta},c}_{\alpha}Z)$ 
for $Y$, $Z \in T \hat{M}$. 
\end{lemma}

\begin{proof}
Put $f(r)=Ar^{ \frac{2}{c} }$. Then 
\[ G^{c}_{\alpha}(Y,Z)
= - f(r) \Omega_{\alpha} (Y,\hat{I}^{\bar{\theta},c}_{\alpha}Z)
                          +2 \varepsilon f(r) \langle \theta, \theta \rangle(Y,Z)
                          +\frac{2 \varepsilon }{c^{2}} \frac{f(r)}{r^{2}}(dr \otimes dr)(Y,Z)
\]
for $Y$, $Z \in T \hat{M}$. Since $Z_{0}^{c}=\varepsilon c r \frac{\partial }{\partial r}$, 
we obtain 
\[ d \hat{\theta}^{c}_{\alpha} (Z_{0}^{c},Z_{\alpha})=\varepsilon cr f^{\prime}(r)
    =\varepsilon cr \cdot \frac{2 A}{c} r^{\frac{2}{c}-1}=2 \varepsilon f(r), \] 
\[ G^{c}_{\alpha}(Z_{0}^{c},\hat{I}^{\bar{\theta},c}_{\alpha}(Z_{\alpha}))
=G^{c}_{\alpha}(Z_{0}^{c},Z_{0}^{c}) 
= c^{2} r^{2} \cdot \frac{2 \varepsilon}{c^{2}} \frac{f(r)}{r^{2}}=2 \varepsilon f(r) \]
and
\[ G^{c}_{\alpha}(Z_{\alpha},\hat{I}^{\bar{\theta},c}_{\alpha}(Z_{0}^{c}))
=- G^{c}_{\alpha}( Z_{\alpha}, Z_{\alpha} ) =-2 \varepsilon f(r). \]
Moreover 
we have
\[ d \hat{\theta}^{c}_{\alpha} (Z_{\beta},Z_{\gamma})
=-\hat{\theta}^{c}_{\alpha}([Z_{\beta},Z_{\gamma}]) 
=-2 \varepsilon f(r) \]
and
\[ G^{c}_{\alpha}(Z_{\beta},\hat{I}^{\hat{\theta},c}_{\alpha}(Z_{\gamma}))
   =-G^{c}_{\alpha}(Z_{\beta},Z_{\beta})=-2 \varepsilon f(r), \]
similarly 
$G^{c}_{\alpha}(Z_{\gamma},\hat{I}^{\bar{\theta},c}_{\alpha}(Z_{\beta}))
=2 \varepsilon f(r)$. 
Finally, we see
\[ d \hat{\theta}^{c}_{\alpha} (Y^{h},Z^{h})
=f(r) (d \theta^{c}_{\alpha})(Y^{h},Z^{h}) 
=f(r) \Omega_{\alpha} (Y^{h},Z^{h}) \]
and
\[ G^{c}_{\alpha}(Y^{h},\hat{I}^{\bar{\theta},c}_{\alpha}Z^{h})
   =f(r) \Omega_{\alpha} (Y^{h},Z^{h}). \]
For other combinations of tangent vectors on $\hat{M}$, both tensors
$d \hat{\theta}_{\alpha}$, $G_{\alpha}$ vanish. 
\end{proof}


We define $\mu^{c} : \hat{M} \to {\mathbb R}^{3}$ by 
\begin{equation}  \label{mm:eq} \mu^{c}(x)
=\hat{\theta}^{c}(\hat{X}_{x})=(\hat{\theta}^{c}_{1}(\hat{X}_{x}), 
\hat{\theta}^{c}_{2}(\hat{X}_{x}), 
\hat{\theta}^{c}_{3}(\hat{X}_{x})) \end{equation}
for $x \in \hat{M}$. 
We calculate some formulae which will be used later 
to determine sufficient  conditions for $\mu^{c}$ to be a hypercomplex moment map.

\begin{lemma}\label{ext_der_theta_cpx}
If $X$ is an affine quaternionic vector field on $(M,Q,\nabla)$, then 
we have 
\[ d \mu^{c}_{\alpha} = - \iota_{\hat{X}} d \hat{\theta}^{c}_{\alpha}.
\] 
\end{lemma}

\begin{proof}
By Lemmas \ref{invariant} and \ref{L_X}, $L_{\hat{X}} \theta_{\alpha}=0$ and 
$[\hat{X}, Z^{c}_{0}]=0$.
It follows 
$L_{\hat{X}} \hat{\theta}^{c}_{\alpha}=0$ and 
$d \mu^{c}_{\alpha}
=  d \iota_{\hat{X}} \hat{\theta}^{c}_{\alpha}
= L_{\hat{X}} \hat{\theta}^{c}_{\alpha} -\iota_{\hat{X}} d \hat{\theta}^{c}_{\alpha} 
=-\iota_{\hat{X}} d \hat{\theta}^{c}_{\alpha}$.
\end{proof}


For the CR-condition for $\mu^{c}$, we have

\begin{lemma}\label{cr_cond}
If $X$ is an affine quaternionic vector field on $(M,Q,\nabla)$, then
we have 
\begin{align}
&(d \mu^{c}_{\alpha} \circ \hat{I}^{\bar{\theta},c}_{\alpha})(Z^{c}_{0}) 
= 0, \label{funda1}\\
&(d \mu^{c}_{1} \circ \hat{I}^{\bar{\theta},c}_{1})(\widetilde{B}) 
= 
(d \mu^{c}_{2} \circ \hat{I}^{\bar{\theta},c}_{2})(\widetilde{B}) 
=(d \mu^{c}_{3} \circ \hat{I}^{\bar{\theta},c}_{3})(\widetilde{B})
\,\, \mbox{for any} \,\, B \in \mathfrak{so}(3), \label{funda2} \\ 
&(d \mu^{c}_{\alpha} \circ \hat{I}^{\bar{\theta},c}_{\alpha})(Y)   
= - A r ^{\frac{2}{c}} \Omega_{\alpha}(\hat{X},\hat{I}^{\hat{\theta},c}_{\alpha}Y), \label{funda3}
\end{align}
for all horizontal vector $Y$. 
\end{lemma}

\begin{proof}
By Lemmas \ref{ext_der_theta} and \ref{ext_der_theta_cpx}, we have 
$d \mu^{c}_{\alpha} \circ \hat{I}^{\bar{\theta},c}_{\alpha}
=G^{c}_{\alpha}(\hat{X}, \, \cdot \,)$. 
Then it is easy to see 
$(d \mu^{c}_{\alpha} \circ \hat{I}^{\bar{\theta},c}_{\alpha})(Z^{c}_{0})
=G^{c}_{\alpha}(\hat{X},Z^{c}_{0})=0$. Since $G^{c}_{1}=G^{c}_{2}=G^{c}_{3}$ 
on $\bar{\cal V} \times T \hat{M}$, we obtain (\ref{funda2}). 
Finally for horizontal vector $Y$ we have 
$(d \mu^{c}_{\alpha} \circ \hat{I}^{\bar{\theta},c}_{\alpha})(Y)
=G^{c}_{\alpha}(\hat{X},Y)=
 - A r ^{\frac{2}{c}} \Omega_{\alpha}(\hat{X},\hat{I}^{\bar{\theta},c}_{\alpha}Y)$.
\end{proof}


For the transversality condition for $\mu^{c}$, 
we state the next lemma, which follows  
from the equation $d \mu^{c}_{\alpha} \circ \hat{I}^{\bar{\theta},c}_{\alpha}
=G^{c}_{\alpha}(\hat{X}, \, \cdot \,)$.

\begin{lemma}\label{trans_con}
We have
\[ (d \mu^{c}_{\alpha} \circ \hat{I}^{\bar{\theta},c}_{\alpha})(\hat{X}) 
= G^{c}_{\alpha}(\hat{X},\hat{X})
= -A r^{\frac{2}{c}} \Omega_{\alpha} (\hat{X},\hat{I}^{\bar{\theta},c}_{\alpha}\hat{X})
                          +2 \varepsilon A r^{\frac{2}{c}} 
                          \langle \theta, \theta \rangle(\hat{X},\hat{X}).
\]
\end{lemma}

By Lemma \ref{cr_cond}, 
$\mu^{c}$ satisfies CR equations (\ref{moment_cr}) if and only if
\[ \Omega_{1}(\hat{X},\hat{I}^{\bar{\theta},c}_{1}Y)
=\Omega_{2}(\hat{X},\hat{I}^{\bar{\theta},c}_{2}Y)
=\Omega_{3}(\hat{X},\hat{I}^{\bar{\theta},c}_{3}Y) \]
holds for all horizontal vector $Y$. 
On the other hand, from the equations (\ref{defB}) and (\ref{defOmega}), 
$\Omega_{\alpha}$ satisfies
\begin{eqnarray}\label{Omega_Ricc}
     2 \varepsilon \Omega_{\alpha}(Y^{h},\hat{I}^{\bar{\theta},c}_{\alpha}Z^{h})
&=&\Omega^{U}_{\alpha}(Y,I_{\alpha}Z) \\
&=&\frac{1}{2(n+1)} \left( (Ric^{\nabla})^{a}(I_{\alpha}Y,I_{\alpha}Z)
                                   -(Ric^{\nabla})^{a}(Y,Z) \right) \nonumber \\
& &-\frac{1}{2n} \left( (Ric^{\nabla})^{s}(I_{\alpha}Y,I_{\alpha}Z)
                                   +(Ric^{\nabla})^{s}(Y,Z) \right) \nonumber\\                            
& &+\frac{2}{n(n+2)} (\Pi_{h} (Ric^{\nabla})^{s} )(Y,Z) \nonumber
\end{eqnarray}
for tangent vectors $Y$ and $Z$ on $M$. 
In particular, if $Ric^{\nabla}$ is $Q$-hermitian, then  
$\Omega_{\alpha}(\, \cdot \,, \hat{I}^{\bar{\theta},c}_{\alpha} \, \cdot \,  )$
does not depend on $\alpha$. 
We see that 
$\mu^{c}$ satisfies the CR equations (\ref{moment_cr}) if  
$Ric^{\nabla}(X,Y)=Ric^{\nabla}(IX,IY)$
for all $Y \in TM$ and $I \in {\cal Z}$. 
Moreover if the vector field $X$ on $M$ is affine quaternionic 
and $\mu^{c}$ satisfies the CR equations, then 
\[ \Omega_{\alpha} (\hat{X},\hat{I}^{\bar{\theta},c}_{\alpha}\hat{X})
(=\Omega_{\beta} (\hat{X},\hat{I}^{\bar{\theta},c}_{\beta}\hat{X})
=\Omega_{\gamma} (\hat{X},\hat{I}^{\bar{\theta},c}_{\gamma}\hat{X}))
=-\frac{\varepsilon }{2(n+2)} (Ric^{\nabla})(X,X) \circ \hat{\pi}. \]
The following statements can be obtained for 
the CR equations and the transversality condition.

\begin{proposition}\label{moment_map}
Let $M$ be a quaternionic manifold with a quaternionic connection $\nabla$ and 
$X$ an affine quaternionic vector field on $(M,Q,\nabla)$. 
Assume that there exists $\nu \in \Gamma(S_{0})$ such that $L_{X} \nu=0$.
If 
\[ Ric^{\nabla}(X,Y)=Ric^{\nabla}(IX,IY) \]
for all $Y \in TM$, $I \in {\cal Z}$ and 
\[ (Ric^{\nabla})(X, X) \circ \hat{\pi}
+4(n+2) \langle \theta, \theta \rangle (\hat{X}, \hat{X})  \] 
does not vanish on $\hat{M}$,
then the map $\mu^{c}=A r^{\frac{2}{c}} \theta_{\alpha} : \hat{M} \to {\mathbb R}^{3}$ 
$(A \neq 0)$ satisfies the CR equations 
and the transversality condition for any $c \neq 0$. 
\end{proposition}

Also we have

\begin{corollary}\label{moment_map_cor}
Let $M$ be a quaternionic manifold with a quaternionic connection $\nabla$ and 
$X$ an affine quaternionic vector field on $(M,Q,\nabla)$. 
Assume that there exists $\nu \in \Gamma(S_{0})$ such that $L_{X} \nu=0$.
If $Ric^{\nabla}$ is $Q$-hermitian and 
\[  (Ric^{\nabla})(X, X) \circ \hat{\pi} 
+4(n+2) \langle \theta, \theta \rangle (\hat{X}, \hat{X}) \] 
does not vanish on $\hat{M}$, 
then the map $\mu^{c}=A r^{\frac{2}{c}} \theta_{\alpha} : \hat{M} \to {\mathbb R}^{3}$ 
$(A \neq 0)$ satisfies the CR equations 
and the transversality condition for any $c \neq 0$. 
\end{corollary}

\section{The proof of the main result}
\setcounter{equation}{0}

In this section, we give the proof of our main result. 
Using the hypercomplex quotient in \cite{J}, 
we can obtain a hypercomplex manifold $M^{\prime}$ with certain properties. 
To show it, the following lemmas are needed.



\begin{lemma}\label{Lie_exp1}
We have $L_{\widetilde{B}} \theta=-\varepsilon [B, \theta]$. 
Moreover 
$L_{Z_{\alpha}} \theta_{\alpha}=0$, $L_{Z_{\alpha}} \theta_{\beta}=2 \varepsilon \theta_{\gamma}$ and 
$L_{Z_{\alpha}} \theta_{\gamma}=-2 \varepsilon \theta_{\beta}$. 
\end{lemma}
\begin{proof}
We see that 
$(L_{\widetilde{B}} \theta)(\widetilde{C})
=-\theta ([\widetilde{B},\widetilde{C}])
=-\varepsilon[B,C]=-\varepsilon[B, \theta(\widetilde{C})]$
and
$(L_{\widetilde{B}} \theta)(Y^{h})
=-\theta ([\widetilde{B},Y^{h}])=0(=-[B,\theta(Y^{h}) ])$. 
For the latter statements, we compute
\[  \sum (L_{Z_{\alpha}} \theta_{i}) e_{i} = L_{Z_{\alpha}} \theta 
  = - \varepsilon [e_{\alpha},\sum \theta_{i} e_{i} ] 
  = - \varepsilon [e_{\alpha}, \theta_{\beta} e_{\beta} ]
  -\varepsilon [e_{\alpha}, \theta_{\gamma} e_{\gamma} ]
  = -2 \varepsilon \theta_{\beta} e_{\gamma} + 2 \varepsilon \theta_{\gamma} e_{\beta}. 
\]
\end{proof}



By Lemma \ref{Lie_exp1}, we have

\begin{lemma}\label{Lie_deri_hat_theta}
It holds that 
$L_{Z_{\alpha}} \hat{\theta}^{c}_{\alpha}=0$, 
$L_{Z_{\alpha}} \hat{\theta}^{c}_{\beta}=2\varepsilon \hat{\theta}^{c}_{\gamma}$ 
and 
$L_{Z_{\alpha}} \hat{\theta}^{c}_{\gamma}=-2\varepsilon \hat{\theta}^{c}_{\beta}$.
\end{lemma}

From Lemma \ref{L_X}, it holds

\begin{lemma}\label{Lie_der_Theta}
If $L_{X} \nabla=0$, then $L_{\hat{X}} d \hat{\theta}^{c}_{\alpha}=0$ 
for $\alpha=1,2,3$. 
\end{lemma}

We have the main theorems in this paper.

\begin{theorem}\label{qh_correspond}
Let $(M,Q)$ be a quaternionic manifold.  
We assume that $Q$ is anti-self-dual when $n=1$.
Moreover assume that ${\rm U}(1)$ acts freely on $M$ preserving $Q$. 
We denote by $X$ the vector field generating the ${\rm U}(1)$-action. 
If $Q$ admits a quaternionic connection $\nabla$
such that $L_{X} \nabla=0$, 
\begin{eqnarray}\label{cr_condition}
Ric^{\nabla}(X,Y)=Ric^{\nabla}(IX,IY) 
\end{eqnarray}
for all $Y \in TM$, $I \in {\cal Z}$ and 
\begin{eqnarray}\label{positive_condition}
(Ric^{\nabla})(X, X) \circ \hat{\pi}
+4(n+2) \langle \theta, \theta \rangle (\hat{X}, \hat{X}) 
\end{eqnarray}
does not vanish on $\hat{M}$, 
then the natural lift $\hat{X}$ generates a free ${\rm U}(1)$-action with 
the moment map $\mu^{c}$ defined by {\rm(}\ref{mm:eq}{\rm)}, 
where $c=-4(n+1)$.  
Then 
the corresponding hypercomplex quotient is 
a hypercomplex manifold 
$(M^{\prime},H=(I^{\prime}_{1},I^{\prime}_{2},I^{\prime}_{3}))$
with an $I^{\prime}_{1}$-holomorphic vector field 
$Z$ such that 
$L_{Z} I^{\prime}_{2}=2 \varepsilon I^{\prime}_{3}$, 
$L_{Z}  I^{\prime}_{3}=-2 \varepsilon I^{\prime}_{2}$. 
Moreover the exact $2$-forms $d \hat{\theta}^{c}_{\alpha}$ on $\hat{M}$ 
induce closed $2$-forms
$\Theta^{\prime}_{\alpha}$ on $M^{\prime}$ which satisfy 
$L_{Z} \Theta^{\prime}_{1}=0$, 
$L_{Z} \Theta^{\prime}_{2}=2 \varepsilon \Theta^{\prime}_{3}$, 
$L_{Z} \Theta^{\prime}_{3}=-2 \varepsilon \Theta^{\prime}_{2}$. 
\end{theorem}

\begin{proof}
Choose $c=-4 (n+1)$. Then $\hat{M}$ is a hypercomplex manifold 
by Theorem \ref{hypercomlex_cone}. 
We can choose a ${\rm U}(1)$-invariant volume form $\nu$ on $M$. 
Then the condition (2) in Lemma \ref{triviality} holds, so 
$\hat{X}$ is tangent to $S$, which means that the results in the previous section 
can be applied. 
Since Proposition \ref{moment_map} and the second statement of 
Lemma \ref{horizontal} hold, 
$M^{\prime}=P/{\rm U}(1)$ is a hypercomplex manifold 
with the induced hypercomplex structure 
$I^{\prime}_{1}$, $I^{\prime}_{2}$, $I^{\prime}_{3}$
by \cite[Proposition 3.1]{J}, where 
$P$ is the level set $(\mu^{c})^{-1}((1,0,0))$. 
Based on the proof of \cite[Proposition 3.1]{J}, take 
$V=\{ v \in TP \mid 
(d \mu^{c}_{\alpha} \circ \hat{I}^{\hat{\theta},c}_{\alpha})(v)=0 \}$. 
Then we see that $TP=V \oplus \langle \hat{X} \rangle$.
In particular, $\pi_{P}^{\ast} T M^{\prime} \cong V$, 
where $\pi_{P}:P \to M^{\prime}$ is 
the quotient map. 
The vector field 
$\hat{I}^{\hat{\theta},c}_{1} Z_{0}^{c}=Z_{1}$ is tangent to $P$, 
since 
\[ Z_{1} \mu^{c}_{\alpha} |_{P} 
= (2 \varepsilon \delta_{2 \alpha} \mu^{c}_{3} 
 -2 \varepsilon \delta_{3 \alpha} \mu^{c}_{2})|_{P} 
\]
by Lemma \ref{Lie_deri_hat_theta}. 
By Lemma \ref{invariant}, $Z_{1}$ is a projectable vector field, 
that is, $Z:=\pi_{P}{}_{\ast}(Z_{1})$ is a vector field on $M^{\prime}$. 
The vector field $Z$ satisfies 
$(L_{Z} I^{\prime}_{\alpha})(U)
=\pi_{P}{}_{\ast}((L_{Z_{1}} \hat{I}^{\hat{\theta},c}_{\alpha})(U_{P}))$, where
$U_{P} \in \Gamma(V)$ is any projectable vector field and 
$U=\pi_{P}{}_{\ast}(U_{P})$ is its projection. 
In fact, 
this can be obtained from 
$I^{\prime}_{\alpha} \circ \pi_{P}{}_{\ast} \circ pr_{V}
= \pi_{P}{}_{\ast} \circ pr_{V}\circ \hat{I}^{\hat{\theta},c}_{\alpha}$, where 
$pr_{V}:T \hat{M}|_{P} \to V$ is the projection with respect to 
the $\hat{I}^{\hat{\theta},c}_{\alpha}$-invariant decomposition 
\[ (T \hat{M})|_{P}=V \oplus  
\langle \hat{X}, \hat{I}^{\hat{\theta},c}_{1}\hat{X}, 
\hat{I}^{\hat{\theta},c}_{2}\hat{X}, \hat{I}^{\hat{\theta},c}_{3}\hat{X} 
 \rangle. \] 
Therefore, by Lemmas \ref{Lie_der_hyp_cpx} and \ref{invariant}, 
we see that $Z$ is a $I^{\prime}_{1}$-holomorphic vector field 
such that 
$L_{Z} I^{\prime}_{2} =2 \varepsilon I^{\prime}_{3}$ and
 $L_{Z} I^{\prime}_{3}=-2 \varepsilon I^{\prime}_{2}$. 
Finally, by Lemma \ref{Lie_der_Theta}, we can define $2$-forms 
$\Theta^{\prime}_{1}$, $\Theta^{\prime}_{2}$, $\Theta^{\prime}_{3}$ 
on $M^{\prime}$ 
by $\Theta^{\prime}_{\alpha}(U,W)=(d \hat{\theta}^{c}_{\alpha})(U_{P},W_{P})$ 
for $U=\pi_{P}{}_{\ast}(U_{P})$ and $W=\pi_{P}{}_{\ast}(W_{P})$. 
It is clear that these forms are closed. 
Finally we see that these forms satisfy the desired conditions 
by Lemma \ref{Lie_deri_hat_theta}. 
\end{proof}

\begin{remark}
{\rm 
In Theorem \ref{qh_correspond}, the same conclusion can be 
obtained under the assumption that the action induced by $\hat{X}$ is free 
instead of the assumption 
that the action induced by $X$ is free.
}
\end{remark}

In the case that $Ric^{\nabla}$ is $Q$-hermitian, 
we have the following theorem.

\begin{theorem}\label{one_parameter_family}
Let $(M,Q)$ be a quaternionic manifold.  
We assume that $Q$ is anti-self-dual when $n=1$.
Moreover assume that ${\rm U}(1)$ acts freely on $M$ preserving $Q$. 
We denote by $X$ the vector field generating the ${\rm U}(1)$-action. 
If $Q$ admits a quaternionic connection $\nabla$
such that $L_{X} \nabla=0$, $Ric^{\nabla}$ is $Q$-hermitian and 
\begin{eqnarray}\label{positive_condition1}
(Ric^{\nabla})(X, X) \circ \hat{\pi}
+4(n+2) \langle \theta, \theta \rangle (\hat{X}, \hat{X}) 
\end{eqnarray}
does not vanish on $\hat{M}$, 
then there exists a $1$-parameter family 
$\{ (M^{\prime c}, H^{c}=(I^{\prime c}_{1},I^{\prime c}_{2},I^{\prime c}_{3}) ) \}_{c \neq 0}$
of hypercomplex manifolds   
with an $I^{\prime c}_{1}$-holomorphic vector field 
$Z^{c}$ on $M^{\prime c}$ such that 
$L_{Z^{c}} I^{\prime c}_{2}=2 \varepsilon I^{\prime c}_{3}$, 
$L_{Z^{c}}  I^{\prime c}_{3}=-2 \varepsilon I^{\prime c}_{2}$. 
Moreover the exact $2$-forms $d \hat{\theta}^{c}_{\alpha}$ on $\hat{M}$ 
give a $1$-parameter family 
$\{ (\Theta^{\prime c}_{1}, \Theta^{\prime c}_{2}, 
\Theta^{\prime c}_{3}) \}_{c \neq 0}$ 
of triplets of closed $2$-forms on $M^{\prime c}$ such that 
$L_{Z^{c}} \Theta^{\prime c}_{1}=0$, 
$L_{Z^{c}} \Theta^{\prime c}_{2}=2 \varepsilon \Theta^{\prime c}_{3}$, 
$L_{Z^{c}} \Theta^{\prime c}_{3}=-2 \varepsilon \Theta^{\prime c}_{2}$
and 
\begin{eqnarray}\label{invariant_condition}
\Theta^{\prime c}_{1}(\, \cdot \, ,I^{\prime c}_{1} \, \cdot \, )=
\Theta^{\prime c}_{2}(\, \cdot \, ,I^{\prime c}_{2} \, \cdot \, )=
\Theta^{\prime c}_{3}(\, \cdot \, ,I^{\prime c}_{3} \, \cdot \, ). 
\end{eqnarray}
\end{theorem}

\begin{proof}
Since $Ric^{\nabla}$ is $Q$-hermitian, the almost hypercomplex structures 
$(\hat{I}^{\hat{\theta}, c}_{1},\hat{I}^{\hat{\theta}, c}_{2},\hat{I}^{\hat{\theta}, c}_{3})$
on $\hat{M}$ are integrable for all $c \neq 0$ by Theorem \ref{hypercomlex_cone}. 
Following the same procedure as in the proof of Theorem \ref{qh_correspond}, 
we obtain the claims with exception of the equation (\ref{invariant_condition}).
The latter equation follows from 
$Q$-hermitian assumption for $Ric^{\nabla}$
using Lemma \ref{ext_der_theta} and (\ref{Omega_Ricc}).  
\end{proof}

The assumption (\ref{positive_condition}) 
is formulated in terms of  
objects on the 
Swann bundle $\hat{M}$. 
We have the following corollary under assumptions formulated directly on $M$.

\begin{corollary}\label{cor_ricci_positive}
Let $(M,Q)$ be a quaternionic manifold.  
We assume that $Q$ is anti-self-dual when $n=1$.
Moreover assume that ${\rm U}(1)$ acts freely on $M$ preserving $Q$. 
We denote by $X$ the vector field generating the ${\rm U}(1)$-action. 
If $Q$ admits a quaternionic connection $\nabla$
such that $L_{X} \nabla=0$, $(Ric^{\nabla})^{s}(X,X)>0$ and  
{\rm (}\ref{cr_condition}{\rm )} 
is satisfied {\rm (}resp. $Ric^{\nabla}$ is $Q$-hermitian{\rm )}, 
then we have the same conclusion as 
Theorem \ref{qh_correspond} 
{\rm (}resp. Theorem \ref{one_parameter_family}{\rm )}.
\end{corollary}

%


We call the correspondence from $(M,Q,X)$ to $(M^{\prime},H,Z)$ 
or to $\{ (M^{\prime c},H^{c},Z^{c}) \}_{c \neq 0}$ 
described in 
Theorems \ref{qh_correspond} and \ref{one_parameter_family}
the 
{\it Quaternionic/Hypercomplex-correspondence}
({\it Q/H-correspondence} for short). 

\paragraph{A relation with Swann's twist construction.}
Now we explain how $M^{\prime}$ considered just as a smooth manifold 
can be related to $M$ by Swann's twist construction. 
Consider the Lie subgroup 
$\mathrm{U}(1)_{Z_{1}}
:=\{ g \in {\rm SO}(3) \mid Ad_{g}e_{1}=e_{1} \}$ of $\mathrm{SO}(3)$, 
which can be identified with 
${\rm U}(1)$. 
Notice this group is different from the group 
$\langle \hat{X} \rangle \cong \mathrm{U}(1)$
generated by $\hat{X}$. 
Then $P=(\mu^{c})^{-1}((1,0,0))=(\mu^{c})^{-1}(e_{1})$ 
is a principal ${\rm U}(1)_{Z_{1}}$-bundle over $\hat{\pi} (P)$ 
with a connection $\iota_{P}^{\ast}(\theta_{1})$,
where $\iota_{P}:P \to \hat{M}$ is the inclusion map from $P$. 
In fact, the calculation 
\[ \mu^{c}(pg) =\hat{\theta}(\hat{X}_{pg})
   =\hat{\theta}(R_{g \ast} \hat{X}_{p})
   =Ad_{g^{-1}} \hat{\theta}(\hat{X}_{p})
   =Ad_{g^{-1}} e_{1}
   =e_{1} \]
for $p \in P$ and $g \in {\rm U}(1)_{Z_{1}}$ 
shows that $P$ is invariant under $\mathrm{U}(1)_{Z_{1}}$. 
In particular, 
$P \cap \hat{\pi}^{-1}(x)$ is a union of circles ($\mathrm{U}(1)_{Z_{1}}$-orbits). 
Since the functions $\theta_{\alpha}(\hat{X})|_{\hat{\pi}^{-1}(x)}$ on 
$\hat{\pi}^{-1}(x) \cong \mathbb{H}^{\ast}/\{\pm 1\}$
are linear in the natural coordinates on $\mathbb{H} \cong \mathbb{R}^{4}$ and
$\hat{\theta}_{\alpha}=Ar^{\frac{2}{c}} \theta_{\alpha} (\hat{X})$, 
we see that the above intersection  
$P \cap \hat{\pi}^{-1}(x)$
is a single circle. 
Recall \cite{Sw} that Swann's twist construction produces a new manifold $M^{\prime}$ 
from a manifold $M$ with the following twist data:   
a vector field $\xi$, 
a two form $F$ and a function $a$ on $M$.  
More precisely, $\xi$ generates a ${\rm U}(1)$-action, $F$ is an invariant closed 
2-form which is the curvature form of a connection form on 
a principal ${\rm U}(1)$-bundle, and 
$a$ is non-vanishing and satisfies $da=-\iota_{\xi} F$.
It was shown in \cite{MS} that the HK/QK-correspondence 
can be described using the twist construction and a so-called
elementary deformation of the metric. 

In the setting of the Q/H-correspondence, 
let $s:U \to P$, $U \subset \hat{\pi}(P)$ 
be a local section. Then we define
a two form $F$ and a function $a$ on $\hat{\pi}(P)$ by  
\begin{align*}
F &:= s^{\ast}(d (\iota_{P}^{\ast} \theta_{1}))
=s^{\ast}(d \theta_{1})=s^{\ast} \Omega_{1}, \\
a &:= s^{\ast} (\theta_{1}(\hat{X}) \circ \iota_{P}))=\theta_{1}(\hat{X}) \circ s, 
         =s^{\ast} (\theta_{1}(\hat{X})). 
\end{align*}
Note that both 
$F$ and $a$ are independent of the choice of $s$. 
Then we have

\begin{proposition}\label{twist}
As a smooth manifold, $M^{\prime}$ obtained by the Q/H-correspondence  
is a twist of $\hat{\pi}(P)$ 
in the sense of {\rm \cite{Sw}}  with the twist data 
$(\xi=X, F, a)$ as above. 
\end{proposition}

\begin{proof}
Since $L_{X} \nabla =0$, we have
\begin{align*}
da &= s^{\ast} (d \iota_{\hat{X}} \theta_{1} )
    =  - s^{\ast} ( d \theta_{1}(\hat{X}, \,\, \cdot \,\,))
    = - \Omega_{1} (\hat{X}, s_{\ast}( \,\, \cdot \,\,)) \\ 
&=  - \Omega_{1} (X^{h}, s_{\ast}( \,\, \cdot \,\,))
    =  - \Omega_{1} (s_{\ast}(X), s_{\ast}( \,\, \cdot \,\,))
    =-F(X, \,\, \cdot \,\,). 
\end{align*}
Also we obtain
$L_{X} F = (\iota_{X} d +d \iota_{X}) F = -dda=0$. 
\end{proof}

Note that the complex structures $I^{\prime}_{\alpha}$ are not
${\cal H}$-related to $I_{\alpha}$ in the sense of \cite{Sw}, because 
the invariant subbundle 
$V \subset T \hat{M}$ does not coincides with  
${\cal H}$ in general.

\section{Examples}
\setcounter{equation}{0}

In this section, we give examples. \\

\noindent
{\it QK/HK-correspondence}: 
When $M$ is a possibly indefinite quaternionic 
K{\"a}hler manifold with non zero scalar curvature,   
we can take the Levi-Civita connection $\nabla$ 
as a quaternionic connection 
and if there exists a non-zero quaternionic Killing vector field $X$ on $M$, 
then we can take $X$ as the affine quaternionic 
vector field in the $Q/H$-correspondence.   
The tensor field (\ref{two_forms}) gives a (pseudo-)hyper-K{\"a}hler metric 
on $\hat{M}$ 
and (\ref{invariant_condition}) 
gives a (pseudo-)hyper-K{\"a}hler metric on $M^{\prime}$
if $\hat{X}$ is time-like or space-like
(see \cite{ACDM}). 
Therefore our Q/H-correspondence is a generalization of 
the QK/HK-correspondence. 
The following example is well-known 
(see \cite{H, GaLa, Sw1} for example).

\begin{example}[The cotangent bundle $T^{\ast} {\mathbb C}P^{n}$ as a 
hyper-K{\"a}hler manifold]
{\rm 
Consider the quaternionic (right-)projective space 
$M={\mathbb H}P^{n}$ with the standard quaternionic structure. 
We can choose the Levi-Civita connection $\nabla$ of 
the standard quaternionic K{\"a}hler metric on 
$M$ as a quaternionic connection. 
Then we see the Swann bundle 
$\hat{M}=({\mathbb H}^{n+1} \backslash \{ 0 \})/\{ \pm 1\} \to M$ 
as a hypercomplex manifold with the hypercomplex structure 
$(\hat{I}^{\bar{\theta},c}_{1},\hat{I}^{\bar{\theta},c}_{2},
\hat{I}^{\bar{\theta},c}_{3})$, where $\bar{\theta}$ is the principal connection 
associated with the Levi-Civita connection.  
Let $X$ be the vector field on $M$ 
which generates the ${\rm U}(1)$-action on $M$ by 
quaternionic affine transformations defined by 
$e^{i \theta} \cdot [z_{0}, \dots, z_{n}] 
:= [e^{i \theta}  z_{0}, \dots, e^{i \theta}  z_{n}]$ 
for $e^{i \theta}  \in \rm{U(1)}$ and $[z_{0}, \dots, z_{n}] \in M$. 
It holds that the Ricci tensor 
$Ric^{\nabla}$ is $Q$-hermitian and $Ric^{\nabla}(X,X)>0$. 
Then  we can apply Corollary \ref{cor_ricci_positive}. 
Note that the action induces the
well-known hyper-K{\"a}hler moment map on $\hat{M}$ 
when $c=1$. 
The hyper-K{\"a}hler metric 
$G^{1}_{\alpha}=G^{1}_{\beta}=G^{1}_{\gamma}$ is given by 
a constant multiple of the standard Euclidean flat metric 
on ${\mathbb H}^{n+1} \backslash \{ 0 \}$. 
Applying the $QK/HK$-correspondence 
(which amounts to taking the hyper-K{\"a}hler quotient of $\hat{M}$ 
with respect to the vector field $\hat{X}$)
to this example yields 
Calabi's hyper-K{\"a}hler structure on 
$T^{\ast} {\mathbb C}P^{n}$. 
}
\end{example}



\noindent
{\it Hypercomplex manifold with the Obata connection}: 
Let $(M,(I_{1},I_{2}, I_{3}))$ be a hypercomplex manifold 
and $\nabla$ its Obata connection on $M$. 
We recall the Obata connection is a canonical torsion-free 
connection preserving the hypercomplex structure \cite{O}.  
In particular, it is a quaternionic connection with respect to 
the quaternionic structure $Q=\langle I_{1},I_{2},I_{3} \rangle$. 
Assume that a vector field $X$ 
with the flow $\{ \varphi_{t} \}_{t \in \mathbb{R}}$ on $M$ is given,  
which generates a free action of 
${\rm U}(1) = \mathbb{R}/2\pi \mathbb{Z}$ on $M$  such that 
\begin{equation} \label{rot:eq} 
L_{X} I_{1}=0,\quad L_{X} I_{2}=2 \varepsilon I_{3}\quad \mbox{and}\quad  
L_{X} I_{3}=-2 \varepsilon I_{2}. 
\end{equation}
Then it holds 
\[
\varphi^{\ast}_{-t} I_{1} = I_{1}, \,
\varphi^{\ast}_{-t} I_{2}
=(\cos (2 \varepsilon t)) I_{2} + (\sin (2 \varepsilon t)) I_{3}, \,
\varphi^{\ast}_{-t} I_{3}
= (-\sin (2 \varepsilon t)) I_{2} + (\cos (2 \varepsilon t)) I_{3}. 
\]
This shows that 
$X$ is a quaternionic vector field for the quaternionic structure 
$Q=\langle I_{1},I_{2},I_{3} \rangle$. 
Since $(\varphi_{-t}^{\ast} \nabla)$ is the Obata connection 
for the hypercomplex structure \linebreak
$(\varphi^{\ast}_{-t} I_{1},\varphi^{\ast}_{-t} I_{2}, \varphi^{\ast}_{-t} I_{3})$,  
$(\varphi_{-t}^{\ast} \nabla)$ is again a quaternionic connection for $Q$. 
By the explicit expression 
of the Obata connection in \cite{AM1}, 
we have 
\[ \frac{d}{dt} (\varphi_{-t}^{\ast} \nabla) = 0, \]
and hence $\varphi_{-t}^{ \ast} \nabla=\nabla$ 
for all $t$. It follows that $L_{X} \nabla=0$. 
Because the Ricci curvature of the Obata connection is 
skew symmetric and $Q$-hermitian by 
Corollary 1.6 in \cite{AM1}, 
we can apply the Q/H-correspondence to $(M,Q,\nabla)$ 
obtaining a hypercomplex manifold $M^{\prime}$. 
The manifolds $M$ and $M^{\prime}$ are related as follows.

\begin{proposition}\label{double}
$M$ is a double covering space of $M^{\prime}$. 
\end{proposition}

\begin{proof}
The hypercomplex structure is a global section 
$s:M \to S$ 
\[ x \mapsto s(x)=(I_{1}(x),I_{2}(x),I_{3}(x)), \] 
and 
defines a global trivialization of the principal 
${\rm SO}(3)$-bundle $S$. Take a $\mathrm{U}(1)$-invariant volume form.  
Since 
\begin{eqnarray*}
\hat{M} &=&  
\{ (x,s(x)g,r) \mid x \in M, \,\, g \in \mathrm{SO}(3), \,\, r >0 \} \\ 
&\cong& M \times {\rm SO}(3) \times {\mathbb R}^{>0} 
\supset M \times {\rm SO}(3) \times \{ 1 \} \cong 
M \times {\rm SO}(3) =S 
\end{eqnarray*}
and 
$(I_{1}, \varphi^{\ast}_{-t} I_{2}, \varphi^{\ast}_{-t} I_{3})
   =(I_{1},I_{2},I_{3}) g_{\varepsilon t}
   =(I_{1},I_{2},I_{3}) g_{t}^{\varepsilon}, $    
where 
\begin{equation} \label{gt:eq} 
g_{t}=   
   \left( \begin{array}{ccc}
                    1 & 0 & 0  \\
                    0 & \cos(2 t) &  -\sin (2 t)\\
                    0 & \sin (2 t) &  \cos (2 t)\\
                    \end{array} 
             \right), 
\end{equation}
we can write 
$\hat{\varphi}_{t}(x, (I_{1},I_{2},I_{3}), r)
=(\varphi_{t}(x), (I_{1},I_{2},I_{3})g_{\varepsilon t}, r)$
and hence $\hat{X}_{s(x)}=X_{s(x)}^{h} +(\widetilde{e}_{1})_{s(x)}$. 
Therefore we see that  
$\hat{X}_{s(x) g^{\varepsilon }}
= X_{x}^{h} + \widetilde{(Ad_{g^{-\varepsilon} }e_{1})}_{s(x)g^{\varepsilon}}$, 
where $g \in {\rm SO}(3)$. Then the moment map 
$\mu^{c}:\hat{M} \to \mathfrak{so}(3)(={\mathbb R}^{3})$ 
on $\hat{M}$ is given by 
\[ \mu^{c}(p)=A r^{\frac{2}{c}} \theta(\hat{X}_{p}) 
=A r^{\frac{2}{c}} g^{-\varepsilon} e_{1} g^{\varepsilon} \]
at any point $p=(x,s(x)g^{\varepsilon},r) \in \hat{M}$. 
The level set 
$P:=(\mu^{c})^{-1}(e_{1})=(\mu^{c})^{-1}((1,0,0))$ is given by 
\[ \{ (x,s(x)g^{\varepsilon},r) \in \hat{M} \mid x \in M, \,\,
A r^{\frac{2}{c}} g^{-\varepsilon} e_{1} g^{\varepsilon} 
=e_{1 } \}. \]
Hence we have 
\begin{equation}\label{level_set} 
P=\{ (x,s(x)g^{\varepsilon}, A^{-\frac{c}{2}}) \in \hat{M}  \mid 
x \in M, \,\, g \in {\rm U}(1) \} \cong M \times {\rm U}(1). 
\end{equation}
We obtain a hypercomplex manifold 
$M^{\prime}=P / \langle \hat{X} \rangle$, where 
$\langle \hat{X} \rangle \cong \mathrm{U}(1)$. 
Define a map $k:M \to M^{\prime}$ by 
$k(x)=\pi_{P} ((x,s(x),A^{-\frac{c}{2}}))$ for each $x \in M$, 
where $\pi_P : P \rightarrow M^{\prime}=P/\langle \hat{X}\rangle$ 
is the quotient map. 
Since 
$k^{-1}(y)=\pi_{P}^{-1}(y) \cap (s(M) \times \{A^{-\frac{c}{2}}\}))$
consists of exactly two points for each $y \in M^{\prime}$ 
by (\ref{gt:eq}), $k$ is a double covering map. 
By (\ref{funda3}) in Lemma \ref{cr_cond}, it holds 
$V_{p}=\{ v \in T_{p}P 
\mid (d \mu^{c}_{\alpha} \circ \hat{I}^{\hat{\theta},c}_{\alpha})(v)=0 \}
={\cal H}_{p}$. 
It follows that $\pi_{P}^{\ast}(TM^{\prime}) \cong {\cal H} |_{P}$, 
where ${\cal H}$ is the horizontal subbundle with respect to 
the Obata connection. 
Since $s_{\ast}(Y)=Y^{h}$ for $Y \in TM$, we have
\[ k_{\ast}(I_{\alpha}Y)
=\pi_{P \ast} (s_{\ast}I_{\alpha}Y)
=\pi_{P \ast} ((I_{\alpha}Y)^{h})
=I^{\prime}_{\alpha} (\pi_{P \ast} (Y^{h}))
=I^{\prime}_{\alpha} (\pi_{P \ast} s_{\ast} Y)
=I^{\prime}_{\alpha} (k_{\ast}Y). 
\]
Therefore $k:M \to M^{\prime}$ is a double covering map 
satisfying $k_{\ast} \circ I_{\alpha}=I^{\prime}_{\alpha} \circ k_{\ast}$. 
\end{proof}

Note that $M^{\prime}$ is obtained by
the twist data $(X,F=0,a=1)$.

\begin{example}\label{swann_ex}
{\rm 
For the Swann bundle $\hat{M}$ of a quaternionic manifold $(M,Q)$, we see that 
$Z_{1}=-\hat{I}_{1}^{\bar{\theta},c} Z_{0}^{c}$ satisfies the conditions 
required above by Lemma \ref{Lie_der_hyp_cpx}. 
So $\hat{M}$ is a double covering 
space of $(\hat{M})^{\prime}$. 
}
\end{example}

\noindent
{\it Quaternionic Hopf manifold}: 
Consider ${\mathbb H}^{n} \cong {\mathbb R}^{4n}$ 
as a right-vector space over the quaternions. 
Set $\tilde{M}:={\mathbb H}^{n} \backslash \{ 0 \}$. 
The standard hypercomplex structure 
$\tilde{H}=(\tilde{I}_{1},\tilde{I}_{2},\tilde{I}_{3})$
on $\tilde{M}$ is defined by 
$\tilde{I}_{1}=R_{i},\tilde{I}_{2}=R_{j},\tilde{I}_{3}=-R_{k}$, where $R_{q}$ is 
the right-multiplication by $q \in {\mathbb H}$. 
The hypercomplex structure $\tilde{H}$
gives a global section 
$s:\tilde{M} \to S(\cong \tilde{M} \times \mathrm{SO}(3))$ 
as in the previous example.  
The corresponding quaternionic structure is denoted by 
$\tilde{Q}=\langle \tilde{I}_{1},\tilde{I}_{2},\tilde{I}_{3} \rangle$. 
Let $\tilde{g}$ be the standard flat hyper-K{\"a}hler metric on $\tilde{M}$, 
$A \in {\rm Sp}(n){\rm Sp}(1)$ and $\lambda > 1$. Then 
$\gamma:=\lambda A$ generates a group $\Gamma=\langle \gamma \rangle$ 
of homotheties 
which acts freely and properly discontinuously 
on the simply connected manifold $(\tilde{M},\tilde{g})$. 
We can identify $\tilde{M}$ with ${\mathbb R} \times S^{4n-1}$ by 
means of the diffeomorphism $v \mapsto (t,v/\| v \| )$, where 
$t=\log \| v \| / \log \lambda$. Under this identification, $\gamma$ 
corresponds to the transformation
\begin{eqnarray}\label{expression_A}
T_{A}: {\mathbb R} \times S^{4n-1} \to {\mathbb R} \times S^{4n-1}, 
\,\,\, (t,v) \mapsto (t+1,Av). 
\end{eqnarray}
The quotient 
$\tilde{M}/ \Gamma \cong ({\mathbb R} \times S^{4n-1}) / 
\langle T_{A} \rangle$ 
is diffeomorphic to $S^{1} \times S^{4n-1}$ 
and inherits a quaternionic structure $Q$ 
and a quaternionic connection $\nabla$, 
both invariant under the centralizer 
$G^{Q}:=Z_{{\rm GL}(n,\mathbb H){\rm Sp}(1)} (\{ \gamma \})$ 
of $\gamma$ in 
${\rm GL}(n,\mathbb H){\rm Sp}(1)$. 
(In particular, if $A \in {\rm Sp}(n)$, 
then $\tilde{M} / \Gamma$ 
inherits a hypercomplex structure $H$ and its Obata connection $\nabla$, 
both invariant under the centralizer 
$G^{H} := Z_{{\rm GL}(n,\mathbb H)} (\{ \gamma \})$ 
of $\gamma$ in 
${\rm GL}(n,\mathbb H)$.) 
In fact, 
the quaternionic structure $\tilde{Q}$ on $\tilde{M}$ is 
${\rm GL}(n,\mathbb H){\rm Sp}(1)$-invariant and 
induces therefore an almost quaternionic structure 
$Q$ on $\tilde{M} / \Gamma$, 
since $\Gamma \subset {\rm GL}(n,\mathbb H){\rm Sp}(1)$.  
Moreover, the Levi-Civita connection connection $\tilde{\nabla}$ 
on $(\tilde{M},\tilde{g})$, 
which coincides with the Obata connection with respect to $\tilde{H}$, 
is invariant under all homotheties of $\tilde{M}$. 
Since $\Gamma$ acts by homotheties, we see that 
$\tilde{\nabla}$ induces a torsion-free connection $\nabla$ on 
$\tilde{M}/ \Gamma$, which preserves $Q$. 
This means that $Q$ is a quaternionic
structure on $\tilde{M} / \Gamma$. 
The group $G^{Q}$ acts on 
$\tilde{M} / \Gamma$ preserving the data $(Q,\nabla)$.
If $A \in {\rm Sp}(n)$, then $\Gamma$ preserves the hypercomplex structure
$\tilde{H}$ on $\tilde{M}$ 
and thus induces a hypercomplex structure $H$ and 
$\tilde{\nabla}$ induces the Obata connection $\nabla$ on 
$(\tilde{M} / \Gamma, H)$.  
The centralizer $G^{H}$ of $\gamma=A$ in ${\rm GL}(n,\mathbb H)$
acts on $\tilde{M} / \Gamma$ preserving $(H,\nabla)$. 
We say that $(\tilde{M} / \Gamma,Q)$ (resp. $(\tilde{M} / \Gamma,H)$)
is a {\it quaternionic {\rm(}resp.\ hypercomplex{\rm)} 
Hopf manifold}. Note that the hypercomplex Hopf manifolds
are sometimes called quaternionic Hopf manifolds 
(see \cite{OP} for example). 
  
Now taking $A=R_{q}$ for some unit quaternion $q \neq \pm1$, 
we have a quaternionic Hopf manifold $M=\tilde{M}/\Gamma$. 
Then we see
$G^{Q}={\rm GL}(n,\mathbb H){\rm U}(1) 
={\mathbb R}^{>0} \times {\rm SL}(n,\mathbb H){\rm U}(1)$,  where  
$\mathrm{U}(1)$ denotes the centralizer of $q$ in $\mathrm{Sp}(1)$. 
Up to an automorphism of $\mathrm{Sp}(1)$, we can assume that 
\[ \mathrm{U}(1) = \{ e^{i\theta} \mid \theta \in \mathbb{R}\}.\] 
We take a subgroup 
${\mathbb R}^{>0} \times {\rm Sp}(n){\rm U}(1)$ of $G^{Q}$, 
which acts on $M$ transitively. 
The isotropy subgroup is given by
$\langle \lambda \rangle \times {\rm Sp}(n-1) \triangle_{{\rm U}(1)}$, 
where $\triangle_{U(1)}$ is a diagonally embedded subgroup of 
${\rm Sp}(n) {\rm U}(1) \subset {\rm Sp}(n) {\rm Sp}(1)$ which is isomorphic to 
${\rm U}(1)$. This has an expression as
\[{\rm Sp}(n-1) \triangle_{{\rm U}(1)} 
=\left\{ \left[  
 \left(
\begin{array}{@{\,}c|ccc@{\,}}
e^{i \theta} & 0 & \cdots & 0 \\
\hline
0 &&&\\
\vdots&\multicolumn{3}{c}{\raisebox{-2pt}[0pt][0pt]{\large $A$}}\\
0 &&&\\
\end{array}
\right), \,\, e^{i \theta } \,\,
\right] 
\,\, \middle| \,\, A \in {\rm Sp}(n-1), e^{i \theta} \in {\rm U}(1)
\right\}
.\] 
As described above, we obtain 
an invariant quaternionic structure on the homogeneous space
\[ M=({\mathbb R}^{>0} / \langle \lambda \rangle) \times
    \frac{{\rm Sp}(n) {\rm U}(1)}{{\rm Sp}(n-1) \triangle_{{\rm U}(1)}}. 
\]

\begin{remark}
{\rm 
In particular, 
for $n=1$, this yields a left invariant quaternionic structure on 
${\rm U}(1) \times {\rm Sp}(1)$. For $n=2$, we obtain 
an invariant quaternionic structure on the homogeneous space
\[ {\rm U}(1) \times 
\frac{{\rm Sp}(2) {\rm U}(1)}{{\rm Sp}(1) \triangle_{{\rm U}(1)}}
=\frac{T^{2} \cdot {\rm Sp}(2)}{{\rm U}(2)}. \]
Note that the homogeneous quaternionic space 
$T^{2} \cdot {\rm Sp}(2) / {\rm U}(2)$ has a finite covering
of the form $(T^{2} \times G)/{\rm U}(2)$, where 
$G$ is a compact semisimple Lie group, namely $\mathrm{Sp}(2)$. 
This presentation is of the form
$(T^{k} \times G)/{\rm U}(2)$ as considered in \cite{J1}. 
}
\end{remark}

Consider the ${\rm U}(1)$-action on $\tilde{M}$ defined by 
the right-multiplication  
by elements of 
${\rm U}(1) (\subset {\mathbb R}^{>0} \times {\rm Sp}(n){\rm U}(1)
\subset G^{Q})$ : 
$z \mapsto z \cdot e^{\varepsilon i t}$ 
$(z \in \tilde{M})$. 
Then the corresponding vector field $\tilde{X}$ satisfies
$
      \tilde{X}_{z}
= \varepsilon z i =  \varepsilon \tilde{I}_{1} z$ 
for $z \in \tilde{M}$. 
Moreover we see that the relations (\ref{rot:eq}) in the previous example hold, 
that is, 
$L_{\tilde{X}} \tilde{I}_{1}=0$, 
$L_{\tilde{X}} \tilde{I}_{2}=2 \varepsilon \tilde{I}_{3}$,
$L_{\tilde{X}} \tilde{I}_{3}=-2 \varepsilon \tilde{I}_{2}$. 
The ${\rm U}(1)$-action preserving the quaternionic structure induces 
one on $M$ and 
$\tilde{X}$ induces the vector field $X$ on $M$ generating 
the latter ${\rm U}(1)$-action on $M$. 
Considering the hypercomplex moment map on the 
Swann bundle $\hat{\tilde{M}}$ (resp. $\hat{M}$) of $\tilde{M}$ (resp. $M$) 
and the level set 
$\tilde{P} \subset \hat{\tilde{M}}$ (resp. $P \subset \hat{M}$) 
of the corresponding moment map, 
we can obtain 
a hypercomplex manifold $\tilde{M}^{\prime}$ (resp. $M^{\prime}$). 
In fact, since
$Ric^{\tilde{\nabla}}=0$ (resp. $Ric^{\nabla}=0$) and 
$\hat{\tilde{X}}$ (resp. $\hat{X}$) is not horizontal, 
the $Q/H$-correspondence can be applied to $\tilde{M}$ (resp. $M$), 
cf (\ref{positive_condition}). 

Now we consider $\tilde{M}_{+}:=\tilde{M}/\{ \pm 1\}$ and 
$M_{+}:=M/\{ \pm 1\}$. 
The quotient maps by the action of the group $\{ \pm 1\}\cong\mathbb{Z}_2$ on the manifolds are denoted by 
$\tilde{\pi}_{+}:\tilde{M} \to \tilde{M}_{+}$ and  
$\pi_{+}:M \to M_{+}$, respectively. 
The induced objects on 
$\tilde{M}_{+}$ and $M_{+}$ are denoted by the same letter. 
We obtain a hypercomplex isomorphism between $\tilde{M}^{\prime}$ and 
$\tilde{M}_{+}$ as follows. 
Define $\tilde{f}:\tilde{M}^{\prime} \to \tilde{M}_{+}$ by 
$\tilde{f}(x)=\tilde{\pi}_{+}(\hat{\tilde{\pi}}(u))$ 
for any $x \in \tilde{M}^{\prime}$, 
where  
$\hat{\tilde{\pi}} : \hat{\tilde{M}} \to \tilde{M}$ is the
bundle projection and 
$u  \in \pi_{\tilde{P}}^{-1}( x ) \cap (s(\tilde{M}) \times \{A^{-\frac{c}{2}}\})$. 
Since $\pi_{\tilde{P}}^{-1}(x) \cap (s(\tilde{M}) \times \{A^{-\frac{c}{2}}\})$
consists of exactly two points of the form 
$\{ (\pm p, \tilde{H}, A^{-\frac{c}{2}}) \}$ 
as we observed in the proof of Proposition \ref{double}, 
$\tilde{f}$ is well-defined. It is easy to see
\begin{eqnarray}\label{commut_1}
\tilde{f} \circ \pi_{\tilde{P}}=\tilde{\pi}_{+} \circ \hat{\tilde{\pi}}
\end{eqnarray}
on $s(\tilde{M}) \times \{A^{-\frac{c}{2}}\}$ by the definition of $\tilde{f}$. 
Furthermore we have 
\begin{eqnarray}\label{commut_2}
\tilde{f} \circ \tilde{k}=\tilde{\pi}_{+},
\end{eqnarray}
where $\tilde{k}:\tilde{M} \to \tilde{M}^{\prime}$ is the double covering 
as in Proposition \ref{double} for $\tilde{M}$. In fact, from (\ref{commut_1}), 
it follows that 
$\tilde{f}(\tilde{k}(x))=\tilde{f}(\pi_{\tilde{P}}(s(x),A^{-\frac{c}{2}}))
   =\tilde{\pi}_{+} (\hat{\tilde{\pi}}(s(x),A^{-\frac{c}{2}}))
   =\tilde{\pi}_{+} (x)$
for all $x \in \tilde{M}$.


\begin{lemma}\label{ExHypcomx}
The map $\tilde{f}: \tilde{M}^{\prime} \to \tilde{M}_{+}$ is an isomorphism 
of hypercomplex manifolds.
\end{lemma}

\begin{proof}
To prove that $\tilde{f}$ is injective, 
let $x_{1}, x_{2} \in \tilde{M}^{\prime}$ such that
$\tilde{f}(x_{1})=\tilde{f}(x_{2})$. 
There exists $y_{a} \in \tilde{M}$ such that $x_{a}=\pi_{\tilde{P}}(s(y_{a}), A^{-\frac{c}{2}})$ 
($a=1,2$). Since $\tilde{f}(x_{1})=\tilde{f}(x_{2})$ and (\ref{commut_1}), we have
$\hat{\tilde{\pi}}(s(y_{1}),A^{-\frac{c}{2}})=\pm \hat{\tilde{\pi}}(s(y_{2}),A^{-\frac{c}{2}})$, 
that is, $y_{1}=\pm y_{2}$, or equivalently 
$y_{1}=\varphi_{0}(y_{2})$ or $y_{1}=\varphi_{\pi}(y_{2})$. 
Therefore, we see that 
$(s(y_{1}),A^{-\frac{c}{2}})=(s(\varphi_{\delta}(y_{2})),A^{-\frac{c}{2}})$, where 
$\delta=0$ or $\pi$. Then we have $x_{1}=\pi_{\tilde{P}}(s(y_{1}), A^{-\frac{c}{2}})
=\pi_{\tilde{P}}(s(y_{2}), A^{-\frac{c}{2}})=x_{2}$, which means $\tilde{f}$ is injective. 
To show that $\tilde{f}$ is surjective, 
let $z \in \tilde{M}_{+}$ and choose $y \in \tilde{M}$ such that 
$z=\tilde{\pi}_{+}(y)$. By (\ref{commut_2}), we obtain 
$z=\tilde{\pi}_{+}(y)=\tilde{f}(\tilde{k}(y))$. 
Hence $\tilde{f}$ is surjective.  
The lift of $v \in T\tilde{M}^{\prime}$ to ${\cal H}$ 
is denoted by $v^{h^{\prime}}$. By 
\[ (I^{\prime}_{\alpha})_{x} (v)
=\pi_{\tilde{P}}{}_{\ast u}((\hat{I}^{\bar{\theta},c}_{\alpha})_{u}(v^{h^{\prime}}))
=\pi_{\tilde{P}}{}_{\ast u} ( ((I_{\alpha})_{\tilde{f}(x)}(\hat{\tilde{\pi}}_{\ast u}(v^{h^{\prime}})))^{h} ), 
\] 
then 
\begin{align*}
\tilde{f}_{\ast x}((I^{\prime}_{\alpha})_{x} (v))
&\stackrel{(\ref{commut_1})}{=} \tilde{\pi}_{+ \ast}(\hat{\tilde{\pi}}_{\ast u} (((I_{\alpha})_{\tilde{f}(x)}
((\hat{\tilde{\pi}}_{\ast u})(v^{h^{\prime}})))^{h})) \\
&= (I_{\alpha})_{\tilde{f}(x)}(\tilde{\pi}_{+ \ast}(\hat{\tilde{\pi}}_{\ast u}(v^{h^{\prime}})))
\stackrel{(\ref{commut_1})}{=}
(I_{\alpha})_{\tilde{f}(x)}( \tilde{f}_{\ast x} (v))
\end{align*}
at each point $x \in \tilde{M}^{\prime}$, 
where $u  \in \pi_{\tilde{P}}^{-1}(x) 
\cap (s(\tilde{M}) \times \{A^{-\frac{c}{2}}\})$. 
This shows that the hypercomplex manifolds $M^{\prime}$ 
and $M$ are isomorphic.  
\end{proof}

Set $F:= \tilde{f}\circ \pi_{\tilde{P}} :\tilde{P}  \to \tilde{M}_{+}$.
Hereafter we will denote the equivalence class 
with respect to the action of a group $K$ 
by $[ \, \cdot \,]_{K}$. 

\begin{lemma}\label{equi} 
We have 
$F(\gamma \cdot y)=\lambda \cdot F(y)$
for all $y \in {\tilde{P}}$. 
\end{lemma}
\begin{proof}
For any point $y=(p,\tilde{H}g,A^{-\frac{c}{2}})$ 
($p \in \tilde{M}$ and $g \in \mathrm{U}(1)$) 
of $\tilde{P}$, we have  
$\gamma \cdot y=(\lambda p i, \tilde{H}g \rho(i),A^{-\frac{c}{2}})$  
by (\ref{level_set}), 
where $\rho:\mathrm{Sp}(1) \to \mathrm{SO}(3)$ 
is the standard double covering. Therefore 
we obtain 
\[ 
\gamma \cdot [y]_{\langle \hat{\tilde{X}} \rangle} 
=[( \lambda p i,\tilde{H}g \rho(i),
A^{-\frac{c}{2}})]_{\langle \hat{\tilde{X}} \rangle}
=[(\pm \lambda p \hat{g}^{-1},
\tilde{H},A^{-\frac{c}{2}})]_{\langle \hat{\tilde{X}} \rangle}, 
\]
where $\hat{g} \in \mathrm{Sp}(1)$ such that $\rho(\hat{g})=g$. 
Then it holds
\begin{align*}  
F(\gamma \cdot y)
=\tilde{f}(\gamma \cdot [y]_{\langle \hat{\tilde{X}} \rangle})
=\tilde{f}([(\pm \lambda p\hat{g}^{-1},
\tilde{H,}A^{-\frac{c}{2}})]_{\langle \hat{\tilde{X}} \rangle}) 
=\tilde{\pi}_{+} (\pm \lambda p\hat{g}^{-1})
=\lambda \tilde{f}([y]_{\langle \hat{\tilde{X}} \rangle})
=\lambda F(y)
\end{align*}
from (\ref{commut_1}). 
\end{proof}

Note that $M^{\prime}$ has an induced $\{ \pm 1 \}$-action, since  
the lifted action of $\{ \pm 1 \}$ to the Swann bundle $\hat{\tilde{M}}$
commutes with 
$\Gamma$ and $\mathrm{U}(1)$. 
Let $\pi^{\prime}_{+} : M^{\prime} \to M^{\prime}_{+}$ be the 
quotient map of the action by $\{ \pm 1 \}$ on $M^{\prime}$.  
We can define a map 
\[ \Phi :M^{\prime}_{+}(=\pi^{\prime}_{+}(M^{\prime})) 
\to \tilde{M}_{+} / \langle \lambda \rangle \]
as follows. 
Take any $x \in M^{\prime}_{+}$. 
Then there exists 
$y \in \tilde{P}$ such that $x=\pi^{\prime}_{+}(\pi_{P}([y]_{\Gamma}))$ and we set 
$\Phi(x):=[F(y)]_{\langle \lambda \rangle}$. 
We shall show that $[F(y)]_{\langle \lambda \rangle}$ is independent of the choice of $y$.
If 
$(x=)\pi^{\prime}_{+}(\pi_{P}([y_{1}]_{\Gamma}))=\pi^{\prime}_{+}(\pi_{P}([y_{2}]_{\Gamma}))$, 
there exist $\delta \in \{ \pm 1 \}$, $g \in \mathrm{U}(1)$ and $l \in \mathbb{Z}$ 
such that $y_{1}=\delta \cdot g \cdot \gamma^{l} \cdot y_{2}$. 
By Lemma \ref{equi} and the definitions of $\tilde{f}$ and $\pi_{\tilde{P}}$,
we see 
\[ F(y_{1})=F(\delta \cdot g \cdot \gamma^{l} \cdot y_{2})=\lambda^{l} F(y_{2}), \] 
which implies $[F(y_{1})]_{\langle \lambda \rangle}
=[F(y_{2})]_{\langle \lambda \rangle}$.  
Moreover we have 

\begin{lemma}\label{isomor}
The map $\Phi :M^{\prime}_{+} 
\to \tilde{M}_{+} / \langle \lambda \rangle$ is an isomorphism.
\end{lemma}
\begin{proof}
To prove that $\Phi$ is injective, let 
$x_{1}$, $x_{2} \in M^{\prime}_{+}=\pi^{\prime}_{+}(M^{\prime})$
such that $\Phi(x_{1})=\Phi(x_{2})$. 
There exists $y_{1}$, $y_{2} \in \tilde{P}$ such that 
$x_{a}=\pi^{\prime}_{+}(\pi_{P}([y_{a}]_{\Gamma}))$ ($a=1,2$). 
Since 
$[F(y_{1})]_{\langle \lambda \rangle}=[F(y_{2})]_{\langle \lambda \rangle}$, 
there exists $l \in \mathbb{Z}$ such that 
$F(y_{1})=\lambda^{l} \cdot F(y_{2})=F(\gamma^{l} \cdot y_{2})$. 
Then we have $\tilde{f}(\pi_{\tilde{P}}(y_{1}))
=\tilde{f}(\pi_{\tilde{P}}(\gamma^{l} \cdot y_{2}))$, so 
there exists $g \in \mathrm{U}(1)$ such that 
$y_{1}=g \cdot \gamma^{l} \cdot y_{2}$. 
Therefore we obtain
\[ x_{1}=\pi^{\prime}_{+}(\pi_{P}([y_{1}]_{\Gamma}))
          =\pi^{\prime}_{+}(\pi_{P}([g \cdot \gamma^{l} \cdot y_{2}]_{\Gamma}))
          =\pi^{\prime}_{+}(\pi_{P}([y_{2}]_{\Gamma}))
          =x_{2}. \]
So $\Phi$ is injective. 
Next we shall show that $\Phi$ is surjective. 
Take any $z \in \tilde{M}_{+} / \langle \lambda \rangle$. 
There exists $y \in \tilde{P}$ such that $z=[ F(y)]_{\langle \lambda \rangle}$.  
Setting $x=\pi^{\prime}_{+}(\pi_{P}([y]_{\Gamma}))$, we have $\Phi(x)=z$, 
which means $\Phi$ is surjective. 
Since the hypercomplex structures are invariant under actions of all groups 
$\mathrm{U}(1) = \langle \hat{X} \rangle 
= \langle \hat{\tilde{X}} \rangle$, 
$\Gamma$, $\langle \lambda \rangle$ and $\{ \pm 1\}$ 
in the argument,  $\Phi$ is a hypercomplex isomorphism. 
\end{proof}
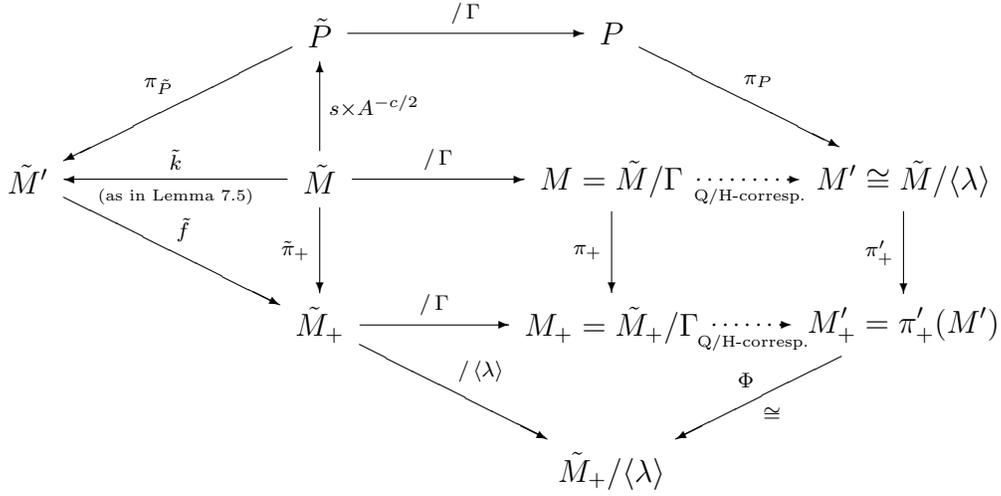
\vspace{-5mm}
\begin{eqnarray*}
  \begin{diagram}
    \node[2]{\tilde{P} } \arrow[1]{e,t}{/ \, \Gamma}
    \arrow{sw,t}{\pi_{\tilde{P}}}
                  \node[1]{P} 
                  \arrow{se,t}{\pi_{P}}\\
    \node[1]{\tilde{M}^{\prime}} 
     \arrow[1]{se,t}{ \tilde{f}}
    \node[1] {\tilde{M}} 
    \arrow[1]{n,r}{s \times A^{-c/2}}
    \arrow[1]{w,tb}{\tilde{k}}
                 {\mbox{\tiny (as in Lemma \ref{ExHypcomx})}}
    \arrow[1]{e,t}{/ \, \Gamma} 
    \arrow[1]{s,l}{\tilde{\pi}_{+} }
    \node[1]{M=\tilde{M} / \Gamma}
    \arrow[1]{e,b,..} {\mbox{\tiny Q/H-corresp.}}
    \arrow[1]{s,l}{\pi_{+} }
    \node[1] {M^{\prime} \cong \tilde{M}/\langle \lambda \rangle}
    \arrow[1]{s,l}{\pi^{\prime}_{+} } \\
    \node[2]{\tilde{M}_{+}}
    \arrow[1]{e,t}{/ \, \Gamma}
    \arrow[1]{se,l}{/ \, \langle \lambda \rangle}
    \node[1]{M_{+}=\tilde{M}_{+}/\Gamma}
    \arrow{e,b,..} {\mbox{\tiny Q/H-corresp.}}
    \node[1]{M_{+}^{\prime}=\pi^{\prime}_{+} (M^{\prime})} 
    \arrow[1]{sw,tb}{\Phi}{\cong}\\
    \node[3]{ \tilde{M}_{+}/\langle \lambda \rangle}
  \end{diagram}
\end{eqnarray*}
Therefore, by Lemma \ref{isomor}, the hypercomplex manifold $M^{\prime}_{+}$
obtained from $M_{+}$ by the Q/H-correspondence is identified with  
$\tilde{M}_{+}/ \langle \lambda \rangle$. 
Since $M^{\prime}$ is the double covering space of $M_{+}^{\prime}$, 
we have
\[ M^{\prime} \cong \tilde{M} /\langle \lambda \rangle. \]
The centralizer $G^{H}$ of $\lambda$ is 
${\rm GL}(n,\mathbb H)={\mathbb R}^{>0} \times {\rm SL}(n,\mathbb H)$ 
and take a subgroup ${\mathbb R}^{>0} \times {\rm Sp}(n)$ of $G^{H}$. 
As we explained, $\tilde{M} / \langle \lambda \rangle$ can be expressed by 
the homogeneous space
\[ \tilde{M} / \langle \lambda \rangle
=({\mathbb R}^{>0} / \langle \lambda \rangle) \times
    \frac{{\rm Sp}(n)}{{\rm Sp}(n-1)}. 
\]
Finally, we summarize the discussion as follows.

\begin{example}\label{Hopf}
{\rm 
The hypercomplex manifold
\[ M^{\prime}=({\mathbb R}^{>0} / \langle \lambda \rangle) \times 
   \frac{{\rm Sp}(n)}{{\rm Sp}(n-1)} 
\]
is obtained by 
the Q/H-correspondence from 
the quaternionic manifold
\[ M=({\mathbb R}^{>0} / \langle \lambda \rangle) \times
    \frac{{\rm Sp}(n) {\rm U}(1)}{{\rm Sp}(n-1) \triangle_{{\rm U}(1)}}. 
\]
(Note that we are  considering the invariant quaternionic 
{\rm (}resp.\ hypercomplex{\rm )} 
structure on $M$ {\rm (}resp. $M^{\prime}${\rm )} described above.)
}
\end{example}

We remark that $M^{\prime}$ does not admit any hyper-K{\"a}hler structure 
for topological reasons, since $M^{\prime}$ is diffeomorphic to 
$S^{1} \times S^{4n-1}$. Therefore our Q/H-correspondence 
yields examples which can not appear in 
the QK/HK-correspondence.





\vspace{8mm}
\noindent
{\bf Acknowledgments.}
Research by the first author is partially funded by the Deutsche Forschungsgemeinschaft (DFG, German Research 
Foundation) under Germany's
Excellence Strategy -- EXC 2121 Quantum Universe -- 390833306. 
This paper was prepared during the second author's stay at 
Universit{\"a}t Hamburg on his sabbatical leave (April 2018 - March 2019). 
He would like to thank  the Department of Mathematics 
of Universit{\"a}t Hamburg for the hospitality. 
His research is partially supported by 
JSPS KAKENHI Grant Number 18K03272.


\noindent
Vicente Cort{\' e}s \\
Department of Mathematics \\
and Center for Mathematical Physics \\
University of Hamburg \\
Bundesstra\ss e 55, \\
D-20146 Hamburg, Germany. \\ 
email:vicente.cortes@uni-hamburg.de \\

\noindent
Kazuyuki Hasegawa \\
Faculty of teacher education \\
Institute of human and social sciences \\
Kanazawa university \\
Kakuma-machi, Kanazawa, \\
Ishikawa, 920-1192, Japan. \\
e-mail:kazuhase@staff.kanazawa-u.ac.jp


\begin{thebibliography}{99}
\bibitem{ACDM}
D. Alekseevsky, V. Cort{\' e}s, M. Dyckmanns and T. Mohaupt, 
{\it Quaternionic K{\"a}hler metrics associated 
with special K{\"a}hler manifolds}, 
J. Geom. Phys. {\bf 92} (2015), 271-287.

\bibitem{ACM} D. Alekseevsky, V. Cort{\' e}s, and T. Mohaupt, 
{\it Conification of K\"ahler and hyper-K\"ahler manifolds}, 
Comm. Math. Phys. {\bf 324} (2013), 637-655. 

\bibitem{AG}
D. Alekseevsky and M. Graev, 
{\it $G$-structures of twistor type and their twistor spaces}, 
J. Geom. Phys. {\bf 10} (1993), 203-229.

\bibitem{AM0}
D. Alekseevsky and S. Marchiafava,
{\it Quaternionic Transformations and
the First Eigenvalues of Laplacian
on a Quaternionic K{\" a}hler Manifold}, 
ESI (The Erwin Schr{\" o}dinger International 
Institute for Mathematical Physics)
{\bf 150} (1994). 

\bibitem{AM1}
D. Alekseevsky and S. Marchiafava,
{\it Quaternionic structures on a manifold and subordinated structures}, 
Ann. Mat. Pura Appl. (4) {\bf 171} (1996), 205-273. 







\bibitem{CDJL}
V. Cort{\' e}s, M. Dyckmanns, M. J\"ungling and D. Lindemann, 
{\it A class of cubic hypersurfaces and 
quaternionic K{\"a}hler manifolds of co-homogeneity one}, 
preprint (arXiv:1701.7882). 

\bibitem{Fu}
S. Fujimura, 
{\it $Q$-connections and their changes on an almost quaternion manifold}, 
Hokkaido Math. J. {\bf 5} (1976), 239-248. 

\bibitem{GaLa}
K. Galicki and B. H. Lawson, 
{\it Quaternionic reduction and quaternionic orbifolds}, 
Math. Ann. {\bf 282} (1988), 1-21. 

\bibitem{Haydys}
A. Haydys, 
{\it Hyper-K{\"a}hler and quaternionic K{\"a}hler manifolds 
with $S^{1}$-symmetries}, 
J. Geom. Phys. {\bf 58} (2008), 293-306.



\bibitem{H}
N. Hitchin, 
{\it Metrics on moduli spaces}, Contemp. Math. 58 (1986), 157-178
(Lefschetz Centennial conference. Proceedings on algebraic geometry).   

\bibitem{J}
D. Joyce, 
{\it The hypercomplex quotient and the quaternionic quotient}, 
Math. Ann. {\bf 290} (1991), 323-340. 

\bibitem{J1}
D. Joyce, 
{\it Compact hypercomplex and quaternionic manifolds}, 
J. Differential Geometry, {\bf 35} (1992), 743-761.



\bibitem{MS}
O. Macia and A. Swann, 
{\it Twist geometry of the $c$-map}, 
Commun. Math. Phys. {\bf 336} (2015), 1329-1357.


\bibitem{O}
M. Obata, 
{\it Affine connections on manifolds with almost complex, quaternion or 
Hermitian structure}, 
Jap. J. Math., {\bf 26} (1956), 43-79.


\bibitem{OP}
L. Ornea and P. Piccinni, 
{\it Locally conformal K{\"a}hler manifold structures in quaternionic geometry}, 
Trans. Amer. Math. Soc., {\bf 349} (1997), 641-355.



\bibitem{P}
H. Pedersen, 
{\it Hypercomplex Geometry}, 
Proceedings of 
the Second Meeting on Quaternionic Structures 
in Mathematics and Physics (2001), 313-320. 


\bibitem{PPS}
H. Pedersen, Y. Poon and A. Swann, 
{\it Hypercomplex structures associated to quaternionic manifolds}, 
Differential Geom. Appl. {\bf 9} (1998), 273-292.




\bibitem{S} 
S. Salamon, 
{\it Differential geometry of quaternionic manifolds}, 
Ann. Scient. {\' E}c. Norm. Sup. {\bf 19} (1986), 31-55.

\bibitem{Sw1}
A. Swann, 
{\it Hyper-K{\"a}hler and quaternionic K{\"a}hler}, 
Math. Ann. {\bf 289} (1991), 421-450. 

\bibitem{Sw} 
A. Swann, 
{\it Twisting hermitian and hypercomplex 
geometries}, 
Duke Math. J. {\bf 155} (2010), 403-431.  




\end{thebibliography}
\end{document}